\documentclass[11pt,reqno]{amsart}
\usepackage[normalem]{ulem}
\PassOptionsToPackage{table}{xcolor}
\usepackage{stmaryrd}
\expandafter\def\csname opt@stmaryrd.sty\endcsname
{only,shortleftarrow,shortrightarrow}
\usepackage{extpfeil}
\usepackage{amsmath,amssymb,amsthm}
\usepackage{aliascnt}
\usepackage{varioref}
\usepackage{hyperref}
\usepackage{thmtools}
\usepackage[capitalize,nameinlink,noabbrev]{cleveref}
\hypersetup{colorlinks=true,urlcolor=blue,citecolor=blue,linkcolor=blue}
\usepackage{courier}
\usepackage{tikz}
\usepackage{tikz-cd}
\usepackage{quiver}
\usetikzlibrary{calc,matrix,arrows,decorations.markings}
\usepackage{array}
\usepackage{color}
\usepackage{cancel}
\usepackage{enumerate}
\usepackage{nicefrac}
\usepackage{listings}
\usepackage{seqsplit}
\usepackage{longtable}
\usepackage{colonequals}
\usepackage{array, multirow}
\newcommand{\F}{\mathbb{F}}
\newcommand{\Q}{\mathbb{Q}}

\newcommand{\R}{\mathbb{R}}
\newcommand{\C}{\mathbb{C}}

\newcommand{\Z}{\mathbb{Z}}

\newcommand{\Aut}{{\rm Aut}}

\newcommand{\PGL}{{\rm PGL}}

\renewcommand{\O}{\mathcal{O}}

\newcommand{\ord}{\operatorname{ord}}

\renewcommand{\P}{\mathbb{P}}

\newcommand{\Jac}{\operatorname{Jac}}

\newcommand{\new}{\operatorname{new}}

\DeclareFontFamily{U}{wncy}{}
    \DeclareFontShape{U}{wncy}{m}{n}{<->wncyr10}{}
    \DeclareSymbolFont{mcy}{U}{wncy}{m}{n}
    \DeclareMathSymbol{\Sh}{\mathord}{mcy}{"58} 

\newtheorem{theorem}{Theorem}[section]

\newtheorem{lemma}[theorem]{Lemma}
\newtheorem{proposition}[theorem]{Proposition}
\newtheorem{corollary}[theorem]{Corollary}
\newtheorem{algorithm}[theorem]{Algorithm}

\theoremstyle{definition}
\newtheorem{definition}[theorem]{Definition}
\newtheorem{example}[theorem]{Example}

\newtheorem{remark}[theorem]{Remark}

\makeatletter
\@namedef{subjclassname@2020}{\textup{2020} Mathematics Subject Classification}
\makeatother

\usepackage{xcolor}

\newcommand{\freddy}[1]{{\color{magenta} \textsf{$\spadesuit\spadesuit\spadesuit$ Freddy: [#1]}}}

\title{Shimura curve Atkin--Lehner quotients of genus at most two}

\begin{document}


\author{\sc Oana Padurariu}
\address{Oana Padurariu \\
Max-Planck-Institut für Mathematik Bonn\\
Germany}
\urladdr{https://sites.google.com/view/oanapadurariu/home}
\email{oana.padurariu11@gmail.com}

\author{\sc Frederick Saia}
\address{Frederick Saia \\
University of Illinois Chicago\\
USA}
\urladdr{https://fsaia.github.io/site/}
\email{fsaia@uic.edu}

\begin{abstract}
    We provide a complete enumeration of all quotients of genus $0, 1$ and $2$ of the Shimura curves $X_0^D(N)$ over $\Q$ by non-trivial subgroups of Atkin--Lehner involutions. For all $1270$ genus $1$ quotients $X$ with $N$ squarefree, we determine the isomorphism class of the Jacobian $X$. For $146$ non-elliptic genus $1$ quotients $X$ and for $405$ bielliptic genus $2$ quotients $X$, we provide a defining equation for $X$. A main tool for us is the theory of {\v{C}}erednik--Drinfeld uniformizations of the curves $X_0^D(N)$, which we implement in wider generality than has previously been done in the literature. 
\end{abstract}

\maketitle


\section{Introduction}

Since Mazur's work on isogenies of elliptic curves over $\Q$ \cite{MazurIsogenies}, there has been sustained interest in classifying those modular curves $X_0(N)$ over $\Q$ which possess infinitely many points of specified degree $d$. The cases of $d=2$ and $d=3$ were respectively handled by Bars \cite{Bars99} (combined with work of Ogg \cite{Ogg74}) and by Jeon \cite{Jeon}. The $d=4$ case was recently handled independently by Hwang--Jeon \cite{HJ24} and by Derickx--Orli\'{c} \cite{DO24}.

Whereas the curves $X_0(N)$ parameterize elliptic curves with additional isogeny structure, the Shimura curves $X_0^D(N)$ over $\Q$, with $\gcd(D,N) = 1$ and with $D>1$ the discriminant of an indefinite quaternion algebra over $\Q$, parameterize abelian surfaces with quaternionic multiplication and additional structure. These Shimura curves have no odd-degree points (see \cref{thm: no_real_points}), so the first degree to consider is $d=2$. In this case, an analogue of Bars' result on infinitude of quadratic points was proven for the curves $X_0^D(N)$ with $D>1$ by the current authors in \cite{PS25}, generalizing a result of Rotger in the $N=1$ case \cite{Rotger02}. 

Given the motivation of classifying torsion structures of abelian surfaces, the same classification problem for quotients of the Shimura curves $X_0^D(N)$ by subgroups of the full Atkin--Lehner group $W_0(D,N) \leq \Aut(X_0^D(N))$ is equally worthy of attention; these quotients parameterize abelian surfaces with \emph{potential} quaternionic multiplication and additional structure. Additionally, these quotients are themselves significant in studying infinitude of points of fixed degree for the curves $X_0^D(N)$, as they provide subcovers which can act as sources of such points. For example, from \cite{PS25} we know that each Shimura curve $X_0^D(N)$ with infinitely many quadratic points has an Atkin--Lehner quotient $X_0^D(N)/\langle w_m \rangle$ of genus $0$ or $1$ with infinitely many rational points.

In this work, we initiate a study of low-degree points on Atkin--Lehner quotients of the Shimura curves $X_0^D(N)$ by focusing on low-genus quotients. Our first main result is a determination of all Atkin--Lehner quotients of genus at most $2$. 

\begin{theorem}\label{theorem: genus_le_2}
There are exactly $3711$ Shimura curves quotients $X_0^D(N)/W$, with $D>1$ and with $W \leq W_0(D,N)$ a non-trivial subgroup, of genus at most $2$. In particular, there are 
\begin{itemize}
\item $779$ curves $X_0^D(N)/W$ having genus $0$,
\item $1352$ curves $X_0^D(N)/W$ having genus $1$, and
\item $1580$ curves $X_0^D(N)/W$ having genus $2$. 
\end{itemize}
\end{theorem}

Each quotient of genus $0$ is either a pointless conic over $\Q$, or has a rational point and is therefore isomorphic to $\mathbb{P}^1_\Q$. Each genus $1$ quotient either has a rational point, and is thus an elliptic curve isomorphic to its Jacobian, or is a non-elliptic curve of genus $1$. We determine Jacobians in the genus $1$ case when $N$ is squarefree, and we determine them up to isogeny when $N$ is not squarefree. 

\begin{theorem}\label{theorem: iso_classes_thm_intro}
Let $X = X_0^D(N)/W$ be a genus $1$ curve with $W \leq W_0(D,N)$ a non-trivial Atkin--Lehner subgroup. 
\begin{enumerate}
\item If $X$ is among the $1270$ such curves with $N$ squarefree, then the isomorphism class of $\Jac(X)$ is as given in \cref{table: genus_1_sqfree_iso_classes}.
\item If $X$ is among the $82$ such curves with $N$ not squarefree, then the isogeny class of $\Jac(X)$ is as given in \cref{table: genus_1_not_sqfree_isog_classes}.
\end{enumerate}
\end{theorem}

From \cref{theorem: genus_le_2} and \cref{theorem: iso_classes_thm_intro}, combined with investigations regarding the existence of rational points on Atkin--Lehner quotients of genera $0$ and $1$, we reach the following result regarding infinitude of rational points. 

\begin{corollary}\label{corollary: infinitude_of_rational_points}
\begin{enumerate}
    \item For each of the $650$ triples $(D,N,W)$ listed in \cref{table: genus_0_rat_pts}, the curve $X_0^D(N)/W$ is isomorphic to $\mathbb{P}^1_\Q$. 
    \item For each of the $537$ triples $(D,N,W)$ listed in \cref{table: genus_1_rat_pts_pos_rank}, the curve $X_0^D(N)/W$ is an elliptic curve with positive rank. 
    \item If $X_0^D(N)/W$ has infinitely many rational points and $(D,N,W)$ is not in \cref{table: genus_0_rat_pts} or \cref{table: genus_1_rat_pts_pos_rank}, then $(D,N,W)$ must be among the $48$ triples listed in \cref{table: genus_0_unknown_rat_pts} or among the $54$ triples listed in \cref{table: genus_1_unknown_rat_pts_pos_rank}.
\end{enumerate}
\end{corollary}
\begin{proof}
Part (1) follows from our determination that the listed curves have rational points, which comes in \cref{rational_points_subsection}. Part (2) follows from rational points investigations for genus $1$ quotients which appear in \cref{rational_points_subsection} and \cref{non_elliptic_rat_pts_section}, along with rank computations in Magma for the isogeny classes of the Jacobians of these curves (as given by \cref{theorem: iso_classes_thm_intro}). 

It follows from Faltings' Theorem that if $X_0^D(N)/W$ has infinitely many rational points, then it must be isomorphic to $\P^1_\Q$ or be an elliptic curve over $\Q$ of positive rank. If this is the case and $(D,N,W)$ does not already appear in one of the tables referenced in parts (1) and (2), then $X_0^D(N)$ must either be
\begin{itemize}
    \item a genus $0$ curve for which we remain unsure of the existence of a rational point, in which case $(D,N,W)$ appears in \cref{table: genus_0_unknown_rat_pts}, or
    \item a genus $1$ curve with positive rank Jacobian for which we remain unsure of the existence of a rational point, in which case $(D,N,W)$ appears in \cref{table: genus_1_unknown_rat_pts_pos_rank}. \qedhere
\end{itemize}
\end{proof}

We also produce models for hundreds of Atkin--Lehner quotients of Shimura curves. To our knowledge, of the models referenced in the following theorem fewer than $70$ have appeared in prior work (including \cite{Kurihara, GR04, GR06, GY17, PS23}). 

\begin{theorem}\label{theorem: eqns_intro_version}
\begin{enumerate}
    \item Each of the curves $X = X_0^D(N)/W$ listed in \cref{table: non_elliptic_eqns} is non-elliptic of genus $1$. For the $4$ triples $(D,N,W)$ in the set 
    \begin{align*}
        \Big\{ &\left(14, 11, \langle w_2, w_{11} \rangle \right), \left(46, 3, \langle w_2, w_3 \rangle \right), \\  &\left(570, 1, \langle w_5, w_{57} \rangle \right), 
    \left(570, 1, \langle w_6, w_{190} \rangle\right) \Big\},
    \end{align*} 
    the curve $X$ admits exactly one of the two models given in \cref{table: non_elliptic_eqns} over $\Q$. Each of the remaining $146$ curves $X$ in \cref{table: non_elliptic_eqns} admits the given model over $\Q$. 

    \item For the $405$ curves $X = X_0^D(N)/W$ listed in \cref{table: genus_2_bielliptic_eqns}, we have that $X$ is bielliptic of genus $2$ and admits the given model over $\Q$. 
\end{enumerate}
\end{theorem}

A main tool in all parts of this study, following the initial enumeration of quotients of genus at most $2$, is the theory of {\v{C}}erednik--Drinfeld $p$-adic uniformizations of the Shimura curves $X_0^D(N)$ for primes $p \mid D$. We implement an algorithm to determine the structure of the dual graph of the model over $\Z_p$, for $p \mid D$, of an Atkin--Lehner quotient $X_0^D(N)/W$ with $N$ squarefree, from which we are able to compute reduction types of such quotients at primes dividing $D$ (see \cref{algorithm: Kodaira_alg}) and to check for the existence of local points at primes dividing $D$ (see \cref{algorithm: local_points}). The former referenced algorithm is used in our proof of \cref{theorem: iso_classes_thm_intro}, the latter is used in determining the existence of rational points for \cref{corollary: infinitude_of_rational_points}, and both are integral in our results on equations of Atkin--Lehner quotients.

Explicit computations using the theory of {\v{C}}erednik--Drinfeld in this context in past works (e.g., \cite{Kurihara, Ogg85, GR06}) seem to have been limited either to level $N=1$ or to quotients by single Atkin--Lehner involutions. It is our hope that our implementation, along with the extensive computations and examples provided in this work, are helpful to readers wishing to learn or use this theory. 

The remainder of this work is organized as follows. 
\begin{itemize}
    \item In \cref{Background_section}, we provide background material on Shimura curves, their Atkin--Lehner quotients, and the theory of {\v{C}}erednik--Drinfeld uniformizations of these curves. 
    \item We state in \cref{Kodaira_section} the algorithms referenced above. 
    \item \cref{genus_le_2_section} features the proof of \cref{theorem: genus_le_2}, as well as a first pass concerning rational points on quotients of genus $0$ and $1$ in \cref{rational_points_subsection}.
    \item We determine isogeny classes of Jacobians genus $1$ in \cref{genus_1_isogeny_class_section}, and subsequently determine isomorphism classes of these Jacobians when $N$ is squarefree in \cref{genus_1_iso_classes_section}. The main results of these two sections combine to give \cref{theorem: iso_classes_thm_intro}.
    \item Models for non-elliptic genus $1$ curves are determined in \cref{non_elliptic_eqns_section}. We also briefly return to the question of rational points here (\cref{non_elliptic_rat_pts_section}) in order to prove several genus $1$ quotients have rational points in a non-constructive manner. We produce equations for bielliptic genus $2$ Atkin--Lehner quotients in \cref{genus_2_bielliptics_section}. \cref{theorem: eqns_intro_version} follows from the main results of these two sections, \cref{theorem: non-elliptic_genus_one_equations} and \cref{theorem: bielliptic_eqns}. 
\end{itemize}

All computations described in this paper were performed in Magma \cite{Magma}. All related code, as well as computed lists, can be found in the Github repository \cite{PSRepo}.


\section*{Acknowledgments} 

We thank Jim Stankewicz for helpful conversation and for sharing his code for computing the dual graph of the special fiber of an Atkin--Lehner quotient $X_0^D(N)/\langle w_p \rangle$ over $\Z_p$. This project was completed in part while F.S.\@ was a visitor at the Max-Planck-Institut f\"{u}r Mathematik in Bonn, and he is much obliged to the institute for its support and to the staff and researchers at MPIM for their hospitality. F.S.\@ also gratefully acknowledges support for this research provided by an AMS-Simons Travel Grant.


\section{Background}\label{Background_section}

\subsection{The Shimura curves $X_0^D(N)$}\label{Shimura_curves_background_subsection}

A \textbf{quaternion algebra} over a field $F$ is a central simple algebra over $F$ of dimension $4$. For example, we always have the split quaternion algebra $M_2(F)$ over $F$. For a quaternion algebra $B$ over $\Q$, we call $B$ \textbf{indefinite} if it is split over the real numbers, i.e., if
\[ B_\infty := B \otimes_{\Q} \R \cong M_2(\R), \]
and otherwise we call $B$ \textbf{definite}. A quaternion algebra $B$ over $\Q$ is determined up to isomorphism by its \textbf{discriminant} $D$, which is the product of primes $p$ such that $B$ is ramified (non-split) over $\Q_p$. This discriminant $D$ is the product of an even or odd number of distinct prime numbers according to whether $B$ is indefinite or definite, respectively. 

Fix $D$ the discriminant of an indefinite quaternion algebra $B$ over $\Q$ and $N$ a positive integer with $\gcd(D,N) = 1$. We let $\mathcal{O}_N$ be a fixed \textbf{Eichler order} in $B$ of level $N$. That is, $\mathcal{O}_N$ is the intersection of two (not necessarily distinct) maximal orders in $B$, and can be defined locally at each finite place of $\Q$ as follows:
\begin{itemize}
    \item If $p \mid D$, then $B_p := B \otimes_{\Q} \Q_p$ is non-split (is a division field) and the localization $\mathcal{O}_{N,p} := \mathcal{O}_N \times_\Z \Z_p$ is taken to be the unique maximal order of $B_p$.
    \item Otherwise, if $p \nmid D$, we fix an isomorphism $\psi_\ell : B_p \to M_2(\Q_p)$ and we take $\mathcal{O}_{N,p}$ such that 
    \[ \psi_p\left(\mathcal{O}_{N,p})\right) = \left \{ \begin{pmatrix} a & b \\ Nc & d \end{pmatrix} \mid a,b,c,d \in \Z_p \right\}. \]
\end{itemize}

Fixing an isomorphism $\psi_\infty : B_\infty \to M_2(\R)$, we have from this map an action of $B_\infty$ on the double half-plane $\Omega := \C \setminus \R$. We have a Riemann surface $X = \mathcal{O}_N^\times \backslash \Omega$ which is determined up to isomorphism by $D$ and $N$ (in particular, does not depend on our choice of Eichler order), and it is a theorem of Shimura \cite[Main Theorem 1]{Sh67} that $X$ has a canonical model defined over $\Q$ which we denote by $X_0^D(N)$. By work of Drinfeld (see \cite{BC91,Buzzard}), $X_0^D(N)$ admits a flat proper model over $\Z$ which is smooth over $\Z\left[\tfrac{1}{DN}\right]$. In particular, $X_0^D(N)$ has good reduction at all primes $p \nmid DN$.

If $D=1$ then this construction recovers the (non-compact, in our setup) coarse moduli scheme $Y_0(N)$ parameterizing cyclic $N$-isogenies of elliptic curves. If $D>1$, then $X_0^D(N)$ is compact and parameterizes abelian surfaces with quaternionic multiplication by the order $\O_N$. We refer to \cite[Proposition 2.4]{Saia24} for other moduli descriptions, including one involving isogenies of QM abelian surfaces, which are equivalent up to isomorphism of the associated coarse moduli scheme $X_0^D(N)$. In this work, while we are highly motivated by questions related to torsion structures of abelian surfaces, these moduli interpretations are not important ingredients to our study.

In the compact case, it is a result of Shimura that these curves have no real points:
\begin{theorem}\cite[Theorem 0]{Sh75}\label{thm: no_real_points}
If $D>1$, then $X_0^D(N)(\R) = \emptyset$. 
\end{theorem}

We have a natural covering map $\pi_N : X_0^D(N) \to X_0^D(1)$ of degree $\psi(N)$ given by the Dedekind psi function:
\[ \psi(N) = N \prod_{\substack{p \mid N \\ \text{prime}}} \left(1+ \frac{1}{p}\right). \]
These covering maps are compatible, in the sense that for $N_1 \mid N_2$ there is a natural map $\pi : X_0^D(N_2) \to X_0^D(N_1)$ of degree $\psi(N_2)/\psi(N_1)$ such that $\pi_{N_1} \circ \pi = \pi_{N_2}$.

We have the following formula for the genus of the curve $X_0^D(N)$ (see, for example, \cite[Thm. 39.4.20]{Voight21}).
\begin{proposition}\label{prop: genus_formula}
For $D>1$ the discriminant of an indefinite quaternion algebra over~$\Q$, for $N$ a positive integer coprime to $D$ and for $k \in \{3,4\}$, define
\[
e_k(D,N) \colonequals \prod_{\substack{p\mid D \\ \text{prime}}} \left(1 - \left(\frac{-k}{p}\right)\right) \prod_{\substack{\ell \parallel N \\ \text{prime}}} \left(1 + \left(\frac{-k}{\ell}\right)\right) \prod_{\substack{\ell^2 \mid N, \\ \ell \text{ prime}}}\delta_\ell(k) ,
\]
where $\left(\frac{\cdot}{\cdot}\right)$ is the Kronecker quadratic symbol, and
\[   
\delta_\ell(k) = 
     \begin{cases}
       2 &\quad\textnormal{if } \left(\frac{-k}{\ell}\right) = 1, \\
       0 &\quad\textnormal{otherwise.}  \\
     \end{cases}
\]
Then, the genus of $X_0^D(N)$ is given by
\[
g(X_0^D(N)) = 1 + \frac{\phi(D)\psi(N)}{12} - \frac{e_4(D,N)}{4} - \frac{e_3(D,N)}{3},
\]
where $\phi$ denotes the Euler totient function.
\end{proposition}

\subsection{Atkin--Lehner involutions}\label{AL_section}

Let $B$ be an indefinite quaternion algebra of discriminant $D$ over $\Q$ and let $\mathcal{O}_N$ be an Eichler order of level $N$ coprime to $D$ in $B$. Consider the subgroup of the normalizer of $\mathcal{O}_N$ in $B$
\[N_{{B}^\times_{> 0}}(\mathcal{O}_N) := \{u \in {B}^\times : \text{nrd}(u) > 0 \text{ and } u^{-1} \mathcal{O}_N u = \mathcal{O}_N \} \] 
 consisting of elements of positive reduced norm. The group of \textbf{Atkin--Lehner involutions} of $X_0^D(N)$ is defined as 
\[ W_0(D,N) := N_{{B}^\times_{> 0}}(\mathcal{O}_N) / \mathbb{Q}^\times \mathcal{O}_N^\times \leq \text{Aut}(X_0^D(N)). \]
Letting $\omega(n)$ denote the number of distinct prime divisors of an integer $n$, $W_0(D,N)$ is a finite abelian $2$-group of the form
\[
W_0(D,N) = \{w_m : m \parallel DN\} \cong \left(\Z/2\Z\right)^{\omega(DN)}.
\]
Let us be a bit more concrete about the quaternion arithmetic here in a way that will be useful later, following \cite{Ogg83} which in turn follows Eichler. Because $B$ is indefinite, $\mathcal{O}_N$ has trivial class number and hence each left $\mathcal{O}_N$-ideal $\mathfrak{a}$ is principal; there exists $\alpha \in B^\times$ such that $\mathfrak{a} = \mathcal{O}_N \cdot \alpha$. The group $W_0(D,N)$ is realized as the group of two-sided ideals of $\mathcal{O}_N$ modulo the ideals of the form $x \mathcal{O}_N$ with $x \in \Q^\times$. For a Hall divisor $m \parallel DN$, we obtain the corresponding involution of $X_0^D(N)$ as follows: there exists a two-sided ideal $\mathfrak{b} = \beta \mathcal{O}_N$ with $\mathfrak{b}^2 = m \cdot \mathcal{O}_N$. Letting $\mathfrak{b}_p := \mathfrak{b} \otimes_\Z \Z_p$ denote the completion of $\mathfrak{b}$ at $p$, we have that $\beta$ is determined locally as follows:
\begin{itemize}
    \item $\mathfrak{b}_p = \mathcal{O}_{N,p}$ if $p \nmid m$,
    \item $\mathfrak{b}_p$ is the unique maximal ideal of $\mathcal{O}_{N,p}$ if $p \mid \gcd(m,D)$, and 
    \item $\mathfrak{b}_p$ is associate to the ideal 
    \[ \begin{pmatrix} 0 & 1 \\ p^{\text{ord}_p(N)} & 0 \end{pmatrix} \cdot \psi_p\left(\mathcal{O}_{N,p}\right) \]
    if $p \mid N$. 
\end{itemize}
The element $\beta \in B^\times$ normalizes $\mathcal{O}_N$, yielding the involution $w_m$ of $X_0^D(N)$ over $\Q$. 

For a subgroup $W \leq W_0(D,N)$, we call the quotient $X_0^D(N)/W$ an \textbf{Atkin--Lehner quotient} of $X_0^D(N)$. We also call the elements of $\Aut\left(X_0^D(N)/W\right)$ induced by the elements of $W_0(D,N)$ Atkin--Lehner involutions.  

The fixed points of Atkin--Lehner involutions are determined by Ogg as being complex multiplication (CM) points by certain imaginary quadratic orders.
\begin{proposition}\cite[p. 283]{Ogg83}\label{prop: Ogg_fixed_pts}
    Let $m>1$ with $m \parallel DN$. The fixed points of the Atkin--Lehner involution $w_m$ acting on $X_0^D(N)$ are points with \textnormal{CM} by the following imaginary quadratic orders:
\[   
R = 
     \begin{cases}
       \Z[i] \text{ and } \Z[\sqrt{-2}] &\quad\textnormal{if } m = 2, \\
       \Z\left[\frac{1+\sqrt{-m}}{2}\right]\text{and } \Z[\sqrt{-m}] &\quad\textnormal{if } m \equiv 3 \;(\bmod \; 4), \\
       \Z[\sqrt{-m}] &\quad\textnormal{ otherwise}.\\
     \end{cases}
\]
\end{proposition}
For each of the orders $R$ listed for $m$ in this proposition, the count of $R$-CM points fixed by $w_m$ is given explicitly by Ogg as a product
\[
h(R) \prod_{\substack{p \mid \frac{DN}{m} \\ \text{prime}}} \nu_p(R,\mathcal{O}_N),
\]
where $h(R)$ is the class number of the order $R$ and $\nu_p(R,\mathcal{O}_N)$ is the number of inequivalent optimal embeddings of the localization $R_p$ of $R$ at $p$ into the localization $(\mathcal{O}_N)_p$  of $\mathcal{O}_N$ at $p$. The latter optimal embedding counts are given by \cite[Theorem 2]{Ogg83}. One can compute the genus of any Atkin--Lehner quotient using \cref{prop: Ogg_fixed_pts} and \cref{prop: genus_formula} with the Riemann--Hurwitz theorem, and such computations are implemented in the file \texttt{quot{\_}genus.m} in \cite{PSRepo}. 

The following result of Ribet is an important tool in our study. Ribet's original result in \cite{Ribet90} contains a hypothesis that $N$ is squarefree, but this hypothesis can be removed by work of Hijikata--Pizer--Shemanske \cite{HPS89a, HPS89b} and so does not appear here. 

\begin{theorem}\cite{Ribet90, BD96, HPS89a, HPS89b}\label{Ribet_isog}
There is an isogeny defined over $\Q$
\[ \Psi \colon \Jac(X_0(DN))^{D-\new} \longrightarrow \Jac(X_0^D(N)) \]
such that, for each Atkin--Lehner involution $w_m(D,N) \in W_0(D,N)$, the corresponding automorphism $\Psi^*(w_m(D,N))$ on $\Jac(X_0(DN))^{D-\textnormal{new}}$ is
\begin{equation}\label{eqn: AL_action_Ribet}
    \Psi^*(w_m(D,N)) = (-1)^{\omega(\gcd(D,m))}w_m(1,DN),
\end{equation} 
where $w_m(1,DN) \in W_0(1,DN)$ denotes the Atkin--Lehner involution corresponding to $m$ on the modular curve $X_0(DN)$.
\end{theorem}

\subsection{{\v{C}}erednik--Drinfeld uniformizations}\label{dual_graphs_section}

For the remainder of this section, we fix $D>1$ the discriminant of an indefinite quaternion algebra $B$ over $\Q$, $N$ a squarefree positive integer with $\gcd(D,N) = 1$, and $p$ a prime divisior of $D$. We recall the main details of the theory of the $p$-adic uniformization of $X_0^D(N)$ and its Atkin--Lehner quotients at primes $p \mid D$ due to {\v{C}}erednik--Drinfeld. The original sources for these results are \cite{Cerednik, Drinfeld}, while our main references for this material are \cite{Kurihara, Ogg83}, \cite[Chapter 3]{BC91}, and \cite[\S 1.7]{BD96}.

For clarity, let us first recall the notion of a graph (with lengths) that we will use, and related constructions, as given in \cite[\S 3]{Kurihara}.
\begin{definition}
A \textbf{graph} $X$ consists of the following data:
\begin{itemize}
    \item a set $\textnormal{Ver}(X)$ of \textbf{vertices} of $X$,
    \item a set $\textnormal{Ed}(X)$ of \textbf{oriented edges} of $X$, 
    \item functions $o, t : \textnormal{Ed}(X) 
    \rightarrow \textnormal{Ver}(X)$, and 
    \item a function 
    \begin{align*} 
    \textnormal{Ed}(X) &\longrightarrow \textnormal{Ed}(X) \\
    y &\longmapsto \overline{y} 
    \end{align*}
    such that $\overline{\overline{y}} = y$ and $o(y) = t(\overline{y})$ for all $y \in \textnormal{Ed}(X)$.
\end{itemize}
For $y \in \textnormal{Ed}(X)$, we call $o(y)$ and $t(y)$ the \textbf{origin} and \textbf{terminal} vertices of $y$, respectively, and we call the oriented edge $\overline{y}$ the \textbf{inverse} of $y$. An \textbf{(unoriented) edge} is a set $\{y,\overline{y}\}$ with $y \in \textnormal{Ed}(X)$. If $\overline{y} = y$, we will refer to both $y$ and $\{y\}$ as a \textbf{half-edge}. 

A \textbf{graph with lengths} is a pair $(X,f)$ such that $X$ is as above and $f$ is a function from $\textnormal{Ed}(X)$ to the set of positive integers such that $f(y) = f(\overline{y})$ for each $y \in \textnormal{Ed}(X)$. We call $f(y)$ the \textbf{length} of the edge $\{y,\overline{y}\}$.
\end{definition}

Note that if $y$ is a half edge, then necessarily $o(y) = t(y)$. In drawing graphs (with lengths), we will distinguish between a half-edge $y$ and an edge $z$ which is not a half edge but satisfies $o(z) = t(z)$ by drawing the former as a straight segment emanating from $o(y)$ and drawing the latter as a loop at $o(z)$, as in \cref{figure: half_edge_fig}.

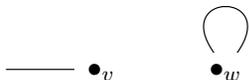
\begin{figure}[h]
    {\centering
    \caption{A vertex $v$ which is the origin of a single half-edge, and a vertex $w$ which is the origin of an oriented edge $z$ with $\overline{z} \neq z$.}\label{figure: half_edge_fig}
 \begin{tikzcd}
	{} & {\bullet_v} & {\bullet_w}
	\arrow[no head, from=1-2, to=1-1]
	\arrow[no head, from=1-3, to=1-3, loop, in=55, out=125, distance=10mm]
\end{tikzcd}}
\end{figure}

If $X$ is a graph, then we denote by $X^*$ the graph obtained from $X$ by removing all half-edges. If $(X,f)$ is a graph with lengths, then we similarly get a graph with lengths $(X^*,f|_{\text{Ed}(X^*})$ via removal of half-edges. 

If $X$ is a graph and $H$ is a group acting on $X$, then we form the \textbf{quotient graph} $X/H$ in the usual way via identification of vertices and oriented edges under the action of $H$. Moreover, if $(X,f)$ is a graph with lengths then we let
\[ H_y := \{h \in H \mid h\cdot y = y\} \]
denote the stabilizer in $H$ of a fixed oriented edge $y$ of $X$ and we have the \textbf{quotient graph with lengths} $(X/H,f_H)$ via 
\[ f_H(y) := f(y) \cdot \#H_y. \]

Given a graph $X$, we construct its \textbf{dual graph} $X^\vee$ as follows: the vertex set of $X^\vee$ is in bijection with the set of (unoriented) edges of $X$, and there will be an edge between two vertices of $X^\vee$ corresponding to edges $e_1$ and $e_2$ of $X$ for each vertex in $X$ that both $e_1$ and $e_2$ are incident to. 

For a graph with lengths $(X,f)$, we construct a graph $\widetilde{X}$, which we call the \textbf{resolution of $X$}, by replacing every edge $y$ of $X$ having length $f(y) > 1$ with a chain of $f(y)$ edges of length $1$ (see \cref{figure: resolution}) via introduction of $f(y)-1$ new vertices.

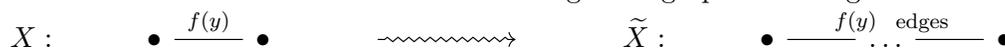
\begin{figure}[h]
    {\centering
    \caption{The resolution of an edge in a graph with lengths}
    \label{figure: resolution}
    \begin{tikzcd}
	{X:} & \bullet & \bullet & {} && {} & {\widetilde{X}:} & \bullet & \ldots & \bullet
	\arrow["{{f(y)}}", no head, from=1-2, to=1-3]
	\arrow[squiggly, from=1-4, to=1-6]
	\arrow["{{f(y)}}"{pos=1}, no head, from=1-8, to=1-9]
	\arrow["{{\text{edges}}}"'{pos=0.9}, no head, from=1-10, to=1-9]
\end{tikzcd}}
\end{figure}

Lastly, for a graph $X$ we construct a graph $X^\text{min}$, which we call the \textbf{minimization of $X$}, by taking $X^*$ (via removing all half-edges from $X$) and then iteratively removing from $X^*$ any edge which is the unique edge incident to a given vertex until no such edges remain. See \cref{figure: minimization} for an example. 

\begin{figure}[h]
    {\centering
    \caption{The minimization of a graph $X$}
    \label{figure: minimization}
    \begin{tikzcd}
	{X:} & \bullet & \bullet & \bullet & {} \\
	{X^\text{min}:} & \bullet & \bullet
	\arrow[no head, from=1-3, to=1-2]
	\arrow[curve={height=-6pt}, no head, from=1-4, to=1-3]
	\arrow[curve={height=6pt}, no head, from=1-4, to=1-3]
	\arrow[shift right, no head, from=1-4, to=1-5]
	\arrow[shift left, no head, from=1-4, to=1-5]
	\arrow[curve={height=-6pt}, no head, from=2-2, to=2-3]
	\arrow[curve={height=6pt}, no head, from=2-2, to=2-3]
\end{tikzcd}}
\end{figure}
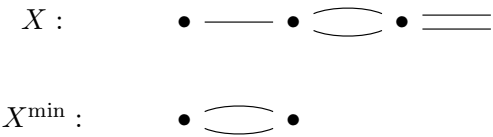

Let $X = X_0^D(N)$, let $\widehat{B}$ and $\widehat{D} = D/p$ denote, respectively, the definite quaternion algebra over $\Q$ which is ramified at the same places as $B$ aside from $p$ and the discriminant of $\widehat{B}$, and let $\widehat{\mathcal{O}}_N$ denote an Eichler order of level $N$ in $\widehat{B}$. Set $h = h\left( \widehat{\mathcal{O}}_N \right)$ as the class number of this order, let $\mathfrak{a}_1, \ldots, \mathfrak{a}_h$ be representative left $\widehat{\mathcal{O}}_N$-ideals of each equivalence class, and let $\widehat{\mathcal{O}}_i$ be the right order corresponding to $\mathfrak{a}_i$ for each $1 \leq i \leq h$. Each $\widehat{\mathcal{O}}_i$ is then an Eichler order of level $N$ in $\widehat{B}$.

We let $\Delta$ denote the Bruhat-Tits tree associated to $\PGL_2(\Q_p)$, which in this setting we realize as having vertex set in correspondence with normalized left $\widehat{\mathcal{O}}_N$-ideals up to homothety. If $x$ is a vertex in $\Delta$ and $\mathfrak{a}_x$ a corresponding left ideal, we let $\widehat{\mathcal{O}}_x$ denote the right order corresponding to $\mathfrak{a}_x$. We have an oriented edge $y$ from a vertex $x_1$ to a vertex $x_2$ if there is a containment $\mathfrak{a}_2 \subseteq \mathfrak{a}_1$ of index $p^2$, in which case the inverse $\overline{y}$ is the edge corresponding to the containment $p\mathfrak{a}_1 \subseteq \mathfrak{a}_2$. In this way, each vertex has $p+1$ neighbor vertices. 

Let 
\[ \widehat{\mathcal{O}}_{N}^{(p)} := \widehat{\mathcal{O}}_N \otimes_\Z \Z[1/p], \]
and consider the quotient group
\[ \Gamma_p := \Z[1/p]^\times \backslash {{\widehat{\mathcal{O}}}_{N}^{(p)}}{^\times} \]
and its index $2$ subgroup
\[ \Gamma := \{\alpha \in \Gamma_p \mid \ord_p(\text{nrd}(\alpha)) \text{ is even}\}. \]
Both of these groups act on $\Delta$ via the action on ideals, and we consider the quotient graphs
\[ G_p := \Delta/\Gamma_p \quad \text{ and } \quad G := \Delta/\Gamma, \]
with $h$ and $2h$ vertices, respectively. Furthermore, we consider each of $G_p$ and $G$ as graphs with lengths by taking $\Delta$ to be a graph with all lengths equal to $1$. That is, the length $f(y)$ of an edge $y$ in $G_p$ or in $G$ is the size of its stabilizer in $\Gamma_p$ or in $\Gamma$, respectively. The stabilizer $\widehat{\mathcal{O}}_x^\times/\{\pm 1\}$ of a vertex $x$ in either of theses graphs has size $f(x) \in \{1, 2, 3, 6, 12\}$, and the stabilizer of an edge is cyclic of order at most $3$ \cite[p. 201]{Ogg83}. For any vertex $x \in G$, an orbit--stabilizer calculation gives
\[ \sum_{\substack{\text{edges } y \\ \text{with } o(y) = x}} \dfrac{f(x)}{f(y)} = p+1.\]

We now describe an action of $W_0(D,N)$, which for expositional purposes we recognize as a product
\[ W_0(D,N) = \widehat{W} \times \langle w_p \rangle \]
with $\widehat{W} = \{w_m \in W_0(D,N) \mid p \nmid m\}$, on $G$. For $w_m \in \widehat{W}$ this is clear -- we take $w_m$ to be given by the action of an element of norm $m$ in $\widehat{O}_N$. The involution $w_p$ is then given by the action of any element of $\Gamma_p$ which is not in $\Gamma$. In this way, we recognize the covering $G \to G_p$ as the bipartite covering of graphs with lengths with respect to the action of $w_p$. This is to say, the $2h$ vertices of $G$ can be partitioned into two sets $V_+ = \{x_{1,+}, \ldots, x_{h,+}\}$ and $V_- = \{x_{1,-}, \ldots, x_{h,-}\}$, where each of $V_+$ and $V_-$ give a set of representatives for the $h$ vertices in $G_p$, such that $w_p(x_{i,+}) = w_p(x_{i,-})$ for each $1 \leq i \leq h$. We have the following useful observations:
\begin{itemize}
    \item Each vertex in $V_+$ is an odd distance from each other vertex in $V_+$, and is an even distance from each vertex in $V_-$. It follows that $G$ contains no half-edges. 
    \item If $w_m \in \widehat{W}$, then $w_m$ permutes both $V_+$ and $V_-$. 
    \item If $p \mid m$, then $w_m$ sends each vertex in $V_+$ to a vertex in $V_-$ (and vice-versa). 
\end{itemize}

Let $W \leq W_0(D,N)$ be an Atkin--Lehner subgroup. Then from the action defined above we get an associated graph with lengths\footnote{We urge the reader to note that the information of $D, N,$ and $p$ are fixed at the start and are not inherent in the notation $G_W$ for these graphs. The amount of decoration we will add to this notation is enough that we believe this to be best, and we will clearly fix these quantities in all examples in order to avoid confusion.} $G_W := G/W$. In particular, we have $G = G_{\{\text{Id}\}}$ and $G_p = G_{\langle w_p \rangle}$. 

We are now prepared to state the main result of this section, as stated in \cite[p. 202]{Ogg85} and \cite[Theorem 1.3, Corollary 1.4]{BD96}.

\begin{theorem}[{\v{C}}erednik--Drinfeld]\label{theorem: CD_uniformization}
Let $D$ be the discriminant of an indefinite quaternion algebra over $\Q$, let $N$ be a squarefree positive integer with $\gcd(D,N) = 1$, and let $p \mid D$ be a prime. Let $W \leq W_0(D,N)$ be an Atkin--Lehner subgroup, and let $G_W$ be as defined above. Then the Atkin--Lehner quotient $X_0^D(N)/W$ over $\Z_p$ is a Mumford curve (c.f. \cite{Mumford}), and the dual graph of its special fiber is given by $\widetilde{G_W^*}$ with $w_p$ inducing the Frobenius automorphism. 
\end{theorem}

Computations of the graphs $G_W$ are achievable using Magma \cite{Magma}, with functionality for the requisite quaternion arithmetic provided in part by the Brandt Modules package by Kohel. Code to compute the data of the graph $G_p$ and the Atkin--Lehner action on $G_p$ was developed by Stankewicz for \cite{Stankewicz} and is available at \cite{StankewiczCode}. We extend this implementation to compute the data of $G_W$ for a general $W \leq W_0(D,N)$ in the file \texttt{dual{\_}graphs.m} in \cite{PSRepo}.


\section{Reduction types and local obstructions}\label{Kodaira_section}

In this section, we describe our implementation of two algorithms which use as a basis the implementation of the theory in \cref{dual_graphs_section}. The first regards reduction types of genus $1$ Atkin--Lehner quotients.

\begin{algorithm}[Kodaira symbols at primes $p \mid D$]\label{algorithm: Kodaira_alg}\phantom{} \\

\textnormal{\textbf{Input:}} A tuple $(D,N,S)$ where 
\begin{itemize}
    \item $D$ is the discriminant of an indefinite quaternion algebra over $\Q$, 
    \item $N$ is a squarefree positive integer with $\gcd(D,N) = 1$, and
    \item $S$ is a set of generators for a subgroup $W \leq W_0(D,N)$ such that the quotient $X_0^D(N)/W$ has genus $1$.
\end{itemize}

\textnormal{\textbf{Output:}} The set of all Kodaira symbols of $\Jac(X_0^D(N)/W)$ at primes $p \mid D$. 
\end{algorithm}

By \cref{theorem: CD_uniformization}, we know that the Kodaira symbol $\Jac(X_0^D(N)/W)$ at a prime $p \mid D$ is $I_n$ for some $n \in \Z^+$. For each fixed $p \mid D$, the steps of \cref{algorithm: Kodaira_alg} to compute the Kodaira symbol at $p$ are as follows:
\begin{enumerate}
    \item Using the material of \cref{dual_graphs_section}, we compute the edge set $E_W$ of $G_W$, along with the action of Atkin--Lehner operators on these edges and the stabilizer orders of each edge. If $e$ is an edge in the edge set $E$ of $G$ and 
    \[ [e] = \{w(e) \mid w \in W\} \]
    is its image in $E_W$, represented as a $W$-orbit of edges from $E$, then the size of the stabilizer of $[e]$ is
    \[ \#\textnormal{Stab}([e]) = \#\textnormal{Stab}(e) \cdot \dfrac{\#W}{\#[e]}. \]
    
    \item We remove half-edges from $E_W$, and we  then remove from $E_W$ any edge which is the unique edge incident to either its origin or terminal vertex in order to obtain $G_W^\text{min}$. We repeat the latter step until no such edges remain, such that our graph is confirmed to be minimal. Let $\widetilde{E_W}$ denote the resulting edge set. 
    
    \item Following step (2), we need only resolve edges in $\widetilde{E_W}$ to get the dual graph of the minimal regular model of $\Jac(X_0^D(N)/W)$ over $\Z_p$. Otherwise put, our Kodaira symbol at $p$ for this Jacobian is $I_n$, where
    \[ n := \sum_{[e] \in \widetilde{E_W}} \#\textnormal{Stab}([e]). \]
\end{enumerate}

\begin{remark}\label{Kodaira_alg_gne1_remark}
If the genus of $X = X_0^D(N)/W$ is greater than $1$, then the steps outlined above still serve to yield the reduction type of $X$ at any prime $p \mid D$. Of course, the possible reduction types for higher genus are more varied, so we leave the statement of \cref{algorithm: Kodaira_alg} as it is for simplicity. We will mainly use the above procedure in the genus $1$ case, but in \cref{genus_2_bielliptics_section} we will briefly consider reduction types in the genus $2$ case.
\end{remark}

Our second algorithm regards the existence of local points at primes dividing $D$. Jordan--Livn\'e \cite[Theorems 5.1, 5.4, 5.6]{JordanLivne} proved results on the existence of $\Q_p$ points on $X_0^D(1)$ for $p \mid D$ , while Ogg \cite[Th\'eor\`eme]{Ogg85} extended these results to Atkin--Lehner quotients $X_0^D(N)/W$ with $\#W \leq 2$. Stankewicz \cite{Stankewicz} also proved quite general results on local points on twists of the Shimura curves $X_0^D(N)$. 

In the works of Ogg and of Jordan--Livne, the authors crucially use Hensel's lemma applied to the situation at hand (see \cite[Lemma 1.1]{JordanLivne}: for a Mumford curve $X$ over $\Z_p$, this states that $X$ has a point over $\Q_p$ if and only if the resolution of the special fiber of $X$ has a smooth point over $\F_p$. We translate this equivalence in our situation using \cref{theorem: CD_uniformization}: fixing a prime $p \mid D$, the quotient $X_0^D(N)/W$ has a point over $\Q_p$ if and only if $\widetilde{G_W^\text{min}}$ has a vertex $x$ so that  
\begin{itemize}
    \item $w_p(x) = x$ (such that $x$ corresponds to a component of the special fiber defined over $\F_p$), and so that 
    \item there are strictly less than $p+1$ oriented edges $y$ with $o(y) = x$ and with $w_p(y) = \overline{y}$ (such that the component corresponding to $x$ contains a smooth point).
\end{itemize}
Using this criterion for local points with our previously described computations of the graphs $\widetilde{G_W}^\text{min}$, we have an algorithmic version of the results of Jordan--Livne and Ogg generalized to arbitrary Atkin--Lehner quotients of $X_0^D(N)/W$ with $N$ squarefree. 

\begin{algorithm}[Local points at primes $p \mid D$]\label{algorithm: local_points}\phantom{} \\

\textnormal{\textbf{Input:}} A tuple $(D,N,S,p)$ where 
\begin{itemize}
    \item $D$ is the discriminant of an indefinite quaternion algebra over $\Q$, 
    \item $N$ is a squarefree positive integer with $\gcd(D,N) = 1$, 
    \item $S$ is a set of generators for a subgroup $W \leq W_0(D,N)$, and
    \item $p \mid D$ is a prime. 
\end{itemize}

\textnormal{\textbf{Output:}} A boolean \texttt{true} if $\left(X_0^D(N)/W\right)(\Q_p) \neq \emptyset$ and \texttt{false} otherwise. 
\end{algorithm}


\section{Atkin--Lehner quotients of genus at most $2$}\label{genus_le_2_section}

\subsection{Proof of \cref{theorem: genus_le_2}}\label{determining_genus_le2_subsection}
We determine here all Atkin--Lehner quotients $X_0^D(N)/W$, with $\gcd(D,N) = 1$ and $D>1$, having genus at most $1$. 

To begin, we recall two results that will allow us restrict attention to a finite list of pairs $(D,N)$. First is a lower bound on the geometric gonality $\gamma_{\C}(X_0^D(N))$ of $X_0^D(N)$, which is a special case of a more general theorem of Abramovich (where we have improved the constant appearing using the best known result on Selberg's eigenvalue conjecture \cite[p. 176]{K03}). 

\begin{theorem}{\cite[Theorem 1.1]{Abramovich96}} \label{theorem: Abr}
For the Shimura curve $X_0^D(N)$, we have
\[
\gamma_{\C}(X_0^D(N)) \ge \frac{975}{8192}\big(g(X_0^D(N))-1\big). 
\]
\end{theorem}

Next, we have an explicit bound on the genus of $X_0^D(N)$ in terms of the product $DN$. 

\begin{lemma}\cite[Lemma 10.6]{Saia24}\label{Lemma: Saia_genus_bound}
Let $\gamma$ denote the Euler--Mascheroni constant. For $D>1$ an indefinite rational quaternion discriminant and $N$ a positive integer coprime to $D$, we have
\[
g(X_0^D(N)) > 1 + \frac{DN}{12}\left( \frac{1}{e^\gamma \log\log(DN) + \frac{3}{\log\log{6}}} \right)- \frac{7\sqrt{DN}}{3}.
\]
\end{lemma}

Let 
\[
X_0^D(N)^* := X_0^D(N)/W_0(D,N)
\]
denote the quotient of $X_0^D(N)$ by the full Atkin--Lehner group. Let $d(m)$ denote the number of divisors of a positive integer $m$. The natural map
\[
X_0^D(N) \rightarrow X_0^D(N)^*
\]
over $\Q$ has degree $2^{\omega(DN)}$, and thus we have, for every algebraic extension $F/\Q$,
\begin{align*} 
\gamma_F(X_0^D(N)) &\leq 2^{\omega(DN)}\gamma_F(X_0^D(N)^*) \\
&\leq d(DN) \gamma_F(X_0^D(N)^*) \\
&\leq 2 \sqrt{DN} \gamma_F(X_0^D(N)^*). 
\end{align*}
\cref{theorem: Abr} then provides
\begin{equation}\label{genus_star_gon_eqn}
\frac{975}{16384\sqrt{DN}}\left(g(X_0^D(N))-1\right) \leq \gamma_{\C}(X_0^D(N)^*) \leq \gamma_F(X_0^D(N)^*), 
\end{equation}
from which \cref{Lemma: Saia_genus_bound} yields
\begin{equation}\label{DN_star_gon_eqn}
\frac{975}{16384} \left( \frac{\sqrt{DN}}{12}\left( \frac{1}{e^\gamma \log\log(DN) + \frac{3}{\log\log{6}}} \right)- \frac{7}{3} \right) \leq \gamma_F(X_0^D(N)^*). 
\end{equation}

Suppose that we have some $W \leq W_0(D,N)$ such that $g\left(X_0^D(N)/W\right) \leq 2$. We then have $g(X_0^D(N)^*) \leq 2$ and $\gamma_{\overline{\Q}}(X_0^D(N)^*) \leq 2$, and so \cref{DN_star_gon_eqn} implies $DN \leq 19226700$. Computing the genus of $X_0^D(N)^*$ for all pairs $(D,N)$ with $DN \leq 19226700$, we determine that among the curves $X_0^D(N)^*$ there are exactly $271$ of genus $0$, exactly $276$ curves of genus $1$, and exactly $289$ of genus $2$. 

Next, we consider general quotients. If $X_0^D(N)^*$ has genus at most $1$, i.e., if $(D,N)$ is among the $836$ pairs we have just computed, then $\omega(DN) \leq 5$. With a bit of care in our enumeration of Atkin--Lehner subgroups, we compute the genus of $X_0^D(N)/W$ for every triple $(D,N,W)$ with $X_0^D(N)^*$ having genus at most $1$ and with $W \leq W_0(D,N)$ in less than an hour of computational time. 

\cref{theorem: genus_le_2} follows from the direct computations described above, with corresponding code found in the file \texttt{genus{\_}le2{\_}quotient{\_}checks.m}.

\subsection{Rational points}\label{rational_points_subsection}

For each Atkin--Lehner quotient $X = X_0^D(N)/W$ of genus $0$ (respectively, of genus $1$), we would next like to determine whether $X$ a rational point and hence is isomorphic to the projective line $\mathbb{P}^1_{\Q}$ (respectively, is an elliptic curve over $\Q$) rather than being a pointless conic (respectively, a non-elliptic genus $1$ curve) over $\Q$. 

To show that $X$ does have a rational point, our first strategy (when $N$ is squarefree) is to show that it has a rational CM point. We does this by enumerating the set $S$ of quadratic CM points on $X_0^D(N)$, which can be done using the work of \cite[\S 5]{GR06} or of \cite{Saia24}, and checking for the existence of a non-trivial involution $w_m \in W$ such that there is a point in $S$ whose image on the quotient $X_0^D(N)/\langle w_m \rangle$, which covers $X$, is a rational point. We do the latter using \cite[Corollary 5.14]{GR06}, which has the required hypothesis that $N$ is squarefree. 

For $X$ of genus $1$, we can also check whether there exists an index $2$ subgroup $W' \leq W$ such that $X_0^D(N)/W'$ has genus $2$, which implies that $X_0^D(N)$ has a rational point by a result of Kuhn \cite[Corollary p. 45]{Kuhn}. Of course, for $X$ having genus $0$ or $1$, we can then also check for a genus $1$ cover of $X$ which is in turn covered by a genus $2$, guaranteeing a rational point on $X$.   

To prove that $X(\Q) = \emptyset$, we first check for local obstructions. If $N$ is squarefree and there exists a prime $p \mid D$ such that $X(\Q_p) = \emptyset$, then we determine this using \cref{algorithm: local_points}. If $\#W = 2$, then we also check whether $X(\R) = \emptyset$ using Ogg's criterion for real points \cite[Proposition 1]{Ogg83}. 

If $X$ has genus $0$, then by the theorem of Hasse--Minkowski we have $X(\Q) \neq \emptyset$ if and only if $X$ has points over $\R$ and over $\Q_p$ for every $p$. If $D$ is even, $N = 1$, and $X$ passed the local obstruction checks described above, then we know it has points over $\Q_p$ for all primes $p$; for $p \nmid D$, the reduction of $X$ over $\F_p$ is a smooth conic, hence isomorphic to $\P^1_{\F_p}$, and then Hensel's lemma implies $X(\Q_p) \neq \emptyset$ since $p \neq 2$. In this setup, if we further have some non-trivial involution $w_m \in W$ such that $X_0^D(N)/\langle w_m \rangle$ has a real point by Ogg's result \cite[Proposition 1]{Ogg83}, then we know $X(\Q) \neq \emptyset$. 

If all of the above fails, then we resort to prior results. For the full quotients $X_0^D(N)^*$ of genus $0$, \cite[Proposition 4.2]{NR15} states that $X_0^D(N)^*$ has a rational CM point. We can also use explicit equations in cases where $X_0^D(N)$ is hyperelliptic or of genus $0$ or $1$, coming from the works \cite{GR06, GY17, PS23}. 

Beyond the above strategies, we have one more technique for proving that genus $1$ curves have rational points. This requires the material of \cref{non_elliptic_eqns_section}, and so we save discussion of this strategy for \cref{non_elliptic_rat_pts_section}. 

Of course, Hasse Principle violations can occur when $X$ has genus $1$. The techniques we use provide no means to detect this unless we already have an explicit equation coming from prior work.

\begin{remark}
We list in the final section of this work all genus $0$ quotients which we prove have rational points (\cref{table: genus_0_rat_pts}), all genus $0$ quotients for which we remain unsure of the existence of rational points (\cref{table: genus_0_unknown_rat_pts}), and genus $1$ quotients with positive rank Jacobian for which we either prove the existence of a rational point (\cref{table: genus_1_rat_pts_pos_rank} or remain unsure of the existence of a rational point (\cref{table: genus_1_unknown_rat_pts_pos_rank}). For brevity, we do not list \emph{all} quotients of genus at most $2$, and for example do not list those curves of genus $0$ for which we prove there are no rational points. Complete lists, however, can be found in \cite{PSRepo}.
\end{remark}

\begin{remark}
In performing rational points checks, we noted that is a typo in \cite[Theorem 3.4]{GR06} for the equation of the involution $w_{10}$ on the curve $X_0^{10}(7)$. The given equation does not define an automorphism on the provided model, and the correct equation is 
\[ w_{10}(x,y) = \left(\frac{2x+1}{x-2},\frac{5y}{(x-2)^2}\right). \]
The first given involution for this curve is also mistakenly called $w_{15}$, whereas $15 \nmid 70$ so this is a clear typo. Said involution is $w_5$. 
\end{remark}

\section{Isogeny classes of Jacobians of genus $1$ Atkin--Lehner quotients}\label{genus_1_isogeny_class_section}

Let $W \leq W_0(D,N)$ be a non-trivial subgroup such that $X = X_0^D(N)/W$ has genus $1$, the Jacobian $J = \Jac(X)$ is an elliptic curve over $\Q$. Using Ribet's isogeny \cref{Ribet_isog} and an explicit decomposition of $\Jac(X_0(DN))^\text{$D$-new}$ (as implemented in Magma), one can in theory determine the isogeny class of $J$. 

However, if the genus of $X_0(DN)$ is large, i.e., if the dimension of $\Jac(X_0(DN))$ is large, the computation of the $D$-new part and its decomposition can make this method infeasible to complete in a reasonable time. Instead, we can compute using modular forms directly by way of modular symbols functionality implemented in Magma. Specifically, we take the Atkin--Lehner decomposition of the new space $M$ of the cuspidal subspace of modular symbols of level $DN$, weight $2$, and sign $-1$, and for each component $V$ we check compatibility of the sign pattern of $V$ with $W$ based on the Hecke action described in \cref{Ribet_isog}. If we find a symbol $V$ which is compatible with $W$, then because $X_0^D(N)/W$ has genus $1$ it must be the case that $V$ is the only compatible component, that $V$ is a one-dimensional space of symbols, and that $X_0^D(N)/W$ is isogenous to the associated elliptic curve. 

The method as described can yield false negatives, though, in the following sense: each elliptic curve over $\Q$ which is isogenous to $X_0^D(N)/W$ may have conductor properly dividing $DN$. It may be the case that $X_0^D(N)/W$ has good reduction at some prime dividing $DN$, i.e., that there is a loss of at least one prime of bad reduction upon taking the quotient of $X_0^D(N)$ by $W$, or $N$ may not be squarefree and we lose a factor whose square divides $N$ in the conductor. This occurs for $348$ genus $1$ quotients $X_0^D(N)/W$. 

We know from \cref{Ribet_isog}, though, that all of the symbols attached to our fixed quotient are $D$-new. Hence, any conductor change only involves divisors of $N$. In this case, one can also see this from direct computation of Kodaira symbols at all primes $p \mid D$ for all genus $1$ quotients $X_0^D(N)/W$ with $N$ squarefree using \cref{algorithm: Kodaira_alg}; no such quotient has Jacobian with good reduction at a prime $p \mid D$. In particular, if $N$ is prime and we observe that there is a loss of prime of bad reduction in taking the quotient by $W$ (by search in the new space of modular symbols of level $DN$ as described above), then we know the correct modular symbol must be new in level $D$. 

This observation helps considerably in reducing computing time in situations where this phenomena occurs, as the following example illustrates. This example also serves to show our method to determine the correct one-dimensional space of symbols in this new space of level $D$ in this situation.

\begin{example}\label{loss_of_N_example}
Consider the genus $1$ curve $X = X_0^{210}(11)/\langle w_2, w_5, w_7, w_{33} \rangle$. The curve $X_0(210\cdot 11) = X_0(561)$ has genus $561$, and so it comes as no surprise that the computation of the decomposition of the abelian subvariety of its Jacobian cut out by the action of $W$ via Ribet's isogeny is intensive; the computation did not complete within three days on our machine. 

We observe by direct computation that there are no components compatible with the sign pattern given by the action of $W$ in the new space of the cuspidal subspace of modular symbols of level $561$, weight $2$, and sign $-1$. We also observe from direct computation using \cref{algorithm: Kodaira_alg} that $J = \Jac(X)$ does not have good reduction at any primes dividing $210$; the symbols at $2, 3, 5,$ and $7$, respectively are $I_4, I_2, I_2$, and $I_4$. Hence, as $11$ is prime, we \emph{must} find that the correct modular symbol and modular form is new in level $210$. 

In searching for a symbol that is compatible with the action of $W$ on this new space of level $D$, one must be careful as the sign pattern imposed by this action at each prime is not necessarily the same as it was in level $DN$. We avoid this difficulty as follows: we simply enumerate \emph{all} one-dimensional components in this space, and correspondingly take all curves in isogeny classes over $\Q$ of conductor $D$, and we then check Kodaira symbols against those computed for $J$. 

In our example, there are $5$ such one-dimensional spaces of symbols which are new in level $210$ and this $37$ isomorphism classes of elliptic curves with conductor $210$ to check -- those in isogeny classes with Cremona labels $210$a, $210$b, $210$c, $210$d, and $210$e. The only curve among these classes with the symbols $I_4, I_2, I_2,$ and $I_4$ at $2, 3, 5,$ and $7$, respectively, is that of Cremona reference $210$c$2$, with equation 
\[ y^2 + xy + y = x^3 + x^2 - 70x - 205, \] 
and so $J$ is isomorphic (not just isogenous, as we aim to show in this section) over $\Q$ to this curve. 
\end{example}

Pleasingly, we find that in all examples where the conductor of $J$ is not $DN$ and where the brute force method using an explicit Jacobian decomposition does not succeed in a reasonable time, we indeed have that $N$ is prime and that we are able to uniquely determine the isomorphism class using Kodaira symbol comparisons as in \cref{loss_of_N_example}. The described computations thus provide the following proposition. 

\begin{proposition}\label{prop: genus_one_jacobian_isogeny_classes}
There are exactly $1270$ curves $X = X_0^D(N)/W$ of genus $1$ with $W \leq W_0(D,N)$ non-trivial and with $N$ squarefree. Each of these curves is listed in \cref{table: genus_1_sqfree_iso_classes}, and in each case $\Jac(X)$ is in the same isogeny class over $\Q$ as the curve with the listed Cremona reference. 

For the $82$ quotients $X = X_0^D(N)/W$ of genus $1$ with $W \leq W_0(D,N)$ nontrivial and with $N$ not squarefree, $\Jac(X)$ is in the isogeny class with Cremona reference given in \cref{table: genus_1_not_sqfree_isog_classes}. 
\end{proposition}


\section{Isomorphism classes of elliptic curve Jacobians}\label{genus_1_iso_classes_section}

Let $X = X_0^D(N)/W$ be a genus $1$ Atkin--Lehner quotient with $W \leq W_0(D,N)$ nontrivial. We know that the elliptic curve $\Jac(X)$ has at worst multiplicative reduction at each $p \mid D$. The following result of Dokchitser--Dokchitser describes the possible behaviors of the Kodaira symbol of $\Jac(X)$ at $p$ under prime-degree isogenies.

\begin{theorem}\cite[Theorem 5.4 (1),(2),(4)]{DD15}\label{Dokchitser_thm}
Let $\varphi: E \to E'$ be an isogeny of elliptic curves over $\Q_p$ of prime degree $\ell$. 
\begin{enumerate}
    \item If $E$ has potentially good reduction and $p \ne \ell$, then the Kodaira types of $E$ and $E'$ are the same.
    \item If $E$ has potentially good ordinary reduction and $p = \ell$, then the Kodaira types of $E$ and $E'$ are the same.
    \item If $E$ has multiplicative reduction, with Kodaira type $I_n$ for $E$, then $E'$ has Kodaira type either $I_{pn}$ or $I_{n/p}$ corresponding to $v(j') = pv(j)$ and $pv(j')=v(j)$, respectively.
\end{enumerate}
\end{theorem}

From \cref{Dokchitser_thm}, we note the following: if we know the Kodaira symbols of $X$ at all primes $p \mid D$ and we know the isogeny class over $\Q$ of $\Jac(X)$ (as determined in \cref{genus_1_isogeny_class_section}), then from this information we can determine the isomorphism class of $\Jac(X)$ \emph{unless} there exist members $E, E'$ of the isogeny class which are isogenous by an isogeny of degree $d$ with $\gcd(d,D) = 1$.

\begin{theorem}\label{theorem: genus_one_jacobian_isomorphism_classes}
Let $X = X_0^D(N)/W$ be a genus $1$ Atkin--Lehner quotient with $W \leq W_0(D,N)$ non-trivial and $N$ squarefree. Then $\Jac(X)$ is isomorphic over $\Q$ to the elliptic curve with Cremona reference given in \cref{table: genus_1_sqfree_iso_classes}. 
\end{theorem}
\begin{proof}
Let $\mathcal{I}$ denote the isogeny class over $\Q$ of $J = \Jac(X)$ as determined in \cref{prop: genus_one_jacobian_isogeny_classes}. Because $N$ is squarefree, we are able to compute the Kodaira symbols of $X$ at all primes $p \mid D$ using \cref{algorithm: Kodaira_alg}. 

If there is a unique element $E$ of $\mathcal{I}$ whose Kodaira symbols at all primes $p \mid D$ match those of $J$, then $J$ is isomorphic to $E$ and we are done. This is the case for all genus one quotients with $N$ squarefree except for the following $5$ triples $(D,N,W)$:
\begin{align*} 
(6,17, \langle w_2, w_3 \rangle) \; , \qquad  (15, 2, \langle w_{30} \rangle) \;, \qquad  (21, 2, \langle w_{21} \rangle)\; , \\
(21, 5, \langle w_{3}, w_{7} \rangle) \;, \quad  \text{ and } \quad  (33,2,\langle w_6, w_{22} \rangle) \; . \phantom{hhh}
\end{align*}
In each of these $5$ listed cases, the Kodaira symbol checks narrow us down to two candidate isomorphism classes (which are necessarily isogenous via a cyclic isogeny of degree coprime to $D$, as per the discussion preceding this theorem). For the first $3$ triples we have that $X_0^D(N)$ is geometrically hyperelliptic, and so we have equations via \cite{GY17, PS23}.

Let $X = X_0^{21}(5)/\langle w_3, w_7 \rangle$. The Kodaira symbol checks for $X$ narrow us down to its Jacobian $J$ being isomorphic to the curve with Cremona reference $105$a$1$ or to that with reference $105$a$3$. The cover $Y = X_0^{21}(5)/\langle w_{21} \rangle$ of $X$ has genus $1$ and we find from the methods in \cref{rational_points_subsection} that is has no rational points. We will compute an equation for this curve in \cref{non_elliptic_eqns_section} (see \cref{table: non_elliptic_eqns}), using which we determine that $J$ has associated Cremona reference $105$a$1$. 

For the elliptic curve $X_0^{33}(2)/\langle w_6, w_{22} \rangle$, we narrow the possibilities down to Cremona references $66$b$1$ and $66$b$3$ via Kodaira symbol checks. We prove that the former reference is correct at the end of this work using material from \cref{genus_2_bielliptics_section}; see \cref{33_2_remark}. 
\end{proof}

\begin{example}
Consider the genus $1$ curves $X = X_0^{15}(7)/\langle w_3 \rangle$ and $Y = X_0^{15}(7)/\langle w_3, w_5 \rangle$. Fixing $D = 15$, $N = 7$, and $p = 3$, we compute the graphs $G_W$ for $W \in \{ \{\text{id}\}, \langle w_3 \rangle, \langle w_3, w_5 \rangle\}$:
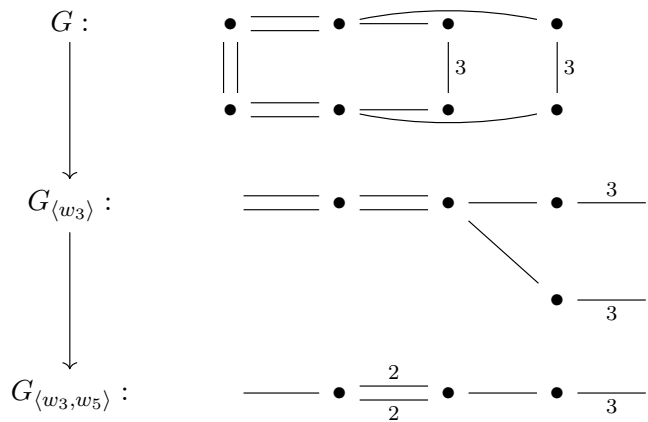
\[ \begin{tikzcd}
	{G:} & \bullet & \bullet & \bullet & \bullet \\
	& \bullet & \bullet & \bullet & \bullet \\
	{G_{\langle w_3 \rangle}:} & {} & \bullet & \bullet & \bullet & {} \\
	&&&& \bullet & {} \\
	{G_{\langle w_3, w_5 \rangle}:} & {} & \bullet & \bullet & \bullet & {}
	\arrow[from=1-1, to=3-1]
	\arrow[shift left, no head, from=1-2, to=1-3]
	\arrow[shift right, no head, from=1-2, to=1-3]
	\arrow[shift left, no head, from=1-2, to=2-2]
	\arrow[shift right, no head, from=1-2, to=2-2]
	\arrow[no head, from=1-3, to=1-4]
	\arrow[curve={height=-6pt}, no head, from=1-3, to=1-5]
	\arrow["3", no head, from=1-4, to=2-4]
	\arrow["3", no head, from=1-5, to=2-5]
	\arrow[shift left, no head, from=2-2, to=2-3]
	\arrow[shift right, no head, from=2-2, to=2-3]
	\arrow[no head, from=2-3, to=2-4]
	\arrow[curve={height=6pt}, no head, from=2-3, to=2-5]
	\arrow[from=3-1, to=5-1]
	\arrow[shift left, no head, from=3-3, to=3-2]
	\arrow[shift right, no head, from=3-3, to=3-2]
	\arrow[shift left, no head, from=3-3, to=3-4]
	\arrow[shift right, no head, from=3-3, to=3-4]
	\arrow[no head, from=3-4, to=3-5]
	\arrow[no head, from=3-4, to=4-5]
	\arrow["3", no head, from=3-5, to=3-6]
	\arrow["3"', no head, from=4-5, to=4-6]
	\arrow[no head, from=5-3, to=5-2]
	\arrow["2", shift left, no head, from=5-3, to=5-4]
	\arrow["2"', shift right, no head, from=5-3, to=5-4]
	\arrow[no head, from=5-4, to=5-5]
	\arrow["3"', no head, from=5-5, to=5-6]
\end{tikzcd} \]
The dual graphs to the minimal regular models of $X$ and $Y$ over $\Z_3$ are then as follows:
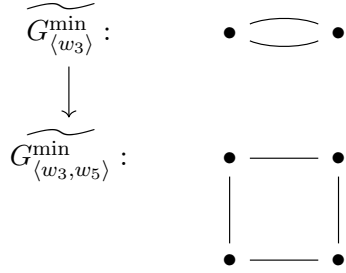
\[ \begin{tikzcd}
	{\widetilde{G_{\langle w_3\rangle}^\text{min}}}: & \bullet & \bullet \\
	{\widetilde{G_{\langle w_3, w_5\rangle}^\text{min}}}: & \bullet & \bullet \\
	& \bullet & \bullet
	\arrow[from=1-1, to=2-1]
	\arrow[curve={height=-6pt}, no head, from=1-2, to=1-3]
	\arrow[curve={height=-6pt}, no head, from=1-3, to=1-2]
	\arrow[no head, from=2-2, to=2-3]
	\arrow[no head, from=2-2, to=3-2]
	\arrow[no head, from=3-2, to=3-3]
	\arrow[no head, from=3-3, to=2-3]
\end{tikzcd} \]
from which we see that $X$ and $Y$ have Kodaira symbols $I_2$ and $I_4$ at $p=3$, respectively. For $p = 5$, we do the same:
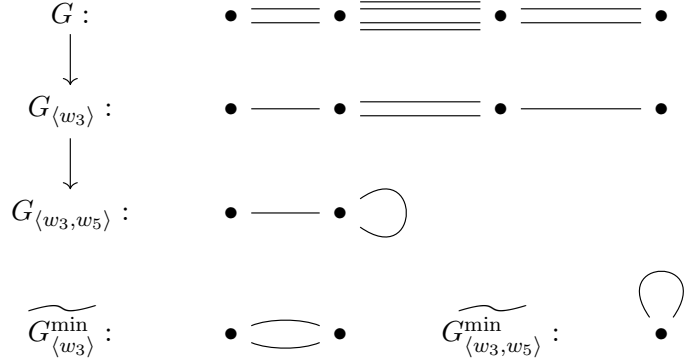
\[\begin{tikzcd}
	{G:} & \bullet & \bullet & \bullet & \bullet \\
	{G_{\langle w_3 \rangle}:} & \bullet & \bullet & \bullet & \bullet \\
	{G_{\langle w_3, w_5 \rangle}:} & \bullet & \bullet \\
	{\widetilde{G_{\langle w_3\rangle}^\text{min}}:} & \bullet & \bullet & {\widetilde{G_{\langle w_3, w_5\rangle}^\text{min}}:} & \bullet
	\arrow[from=1-1, to=2-1]
	\arrow[shift left, no head, from=1-2, to=1-3]
	\arrow[shift right, no head, from=1-2, to=1-3]
	\arrow[shift left=2, no head, from=1-3, to=1-4]
	\arrow[shift right=2, no head, from=1-3, to=1-4]
	\arrow[shift left, no head, from=1-3, to=1-4]
	\arrow[shift right, no head, from=1-3, to=1-4]
	\arrow[shift left, no head, from=1-4, to=1-5]
	\arrow[shift right, no head, from=1-4, to=1-5]
	\arrow[from=2-1, to=3-1]
	\arrow[no head, from=2-2, to=2-3]
	\arrow[shift left, no head, from=2-3, to=2-4]
	\arrow[shift right, no head, from=2-3, to=2-4]
	\arrow[no head, from=2-4, to=2-5]
	\arrow[no head, from=3-2, to=3-3]
	\arrow[no head, from=3-3, to=3-3, loop, in=325, out=35, distance=10mm]
	\arrow[curve={height=-6pt}, no head, from=4-2, to=4-3]
	\arrow[curve={height=-6pt}, no head, from=4-3, to=4-2]
	\arrow[no head, from=4-5, to=4-5, loop, in=55, out=125, distance=10mm]
\end{tikzcd}\]
finding that $X$ and $Y$ have Kodaira symbols $I_2$ and $I_1$ at $p=5$, respectively. From \cref{prop: genus_one_jacobian_isogeny_classes}, we know that $\Jac(X)$ and $\Jac(Y)$ are both in the isogeny class with Cremona reference $105$a, which has $4$ elements. From the above symbols computations, we determine that $\Jac(X)$ has reference $105$a$2$ while $\Jac(Y)$ has reference $105$a$4$ (in fact, the symbols at $p=3$ suffice in this example). From the work described in \cref{rational_points_subsection}, we know that $X$ has no rational points as it lacks points over $\Q_5$, and we will determine an equation for $X$ in \cref{non_elliptic_eqns_section} (see \cref{table: non_elliptic_eqns}). The curve $Y$ is seen to have a rational point as its cover $X_0^{15}(7)/\langle w_{15} \rangle$ has a rational $-7$-CM point, so $Y \cong \Jac(Y)$ over $\Q$. 
\end{example}

\begin{example}
The curve $X = X_0^6(17)/\langle w_2, w_3 \rangle $ is an elliptic curve over $\Q$, and we compute the isomorphism class to be that of Cremona reference $102$b in \cref{prop: genus_one_jacobian_isogeny_classes}. Comparison of the computed Kodaiara symbols of $X$ with those of all members of this isogeny class does not determine its isomorphism class; both the curve with Cremona reference $102$b$5$ and that with reference $102$b$6$, which are isogenous by a $4$-isogeny over $\Q$, have the correct symbols of $I_1$ and $I_2$ at the primes $2$ and $3$, respectively. Since the curve $X_0^6(17)$ is hyperelliptic, though, one finds in \cite{GY17} an equation for $X_0^6(17)$ and in \cite{PS23} an equation for $X$, from which we know $X$ is isomorphic to the curve with Cremona reference $102$b$6$. 
\end{example}

\begin{remark}
In \cref{Appendix: dual graphs}, we draw for all $129$ genus $1$ Atkin--Lehner quotients of the form X = $X_0^D(N)/\langle w_m \rangle$, with $w_m \in W_0(D,N)$ non-trivial, the associated graph $G_{\langle w_m \rangle}$ and the Kodaira symbol of $\Jac(X)$ at each prime $p \mid D$. It is our hope that this wealth of explicit examples will be useful to readers.
\end{remark}

\begin{remark}\label{GY_genus_1_correction}
In \cite{GY17}, the equations provided for the Atkin--Lehner involutions $w_{22}$ and $w_{33}$ (and hence for $w_{2}$ and $w_{3}$ and for $w_{66}$ and $w_{11}$ as well) on the curve $X_0^6(11)$ are swapped. One can see the equations are incorrect as stated by taking the quotients by these involutions, as done in \cite{PS23}, and seeing that the isomorphism classes of their Jacobians are swapped. Comparison with \cite{NR15} confirms that the model in \cite{GY17} for the curve $X_0^{6}(11)$ is correct; it is just that the equations for these involutions are swapped. Similarly, we note the following:
\begin{itemize}
    \item The equations in \cite{GY17} for the involutions $w_{10}$ and $w_{13}$ (and hence also for $w_{5}$ and $w_{26}$ on the curve $X_0^{10}(13)$ are swapped. 
    \item The equations in \cite{GY17} for the involutions $w_{35}$ and $w_{10}$ on the curve $X_0^{14}(5)$ are swapped. In this case, the error results in involutions of genus $1$ being swapped with involutions of genus $2$, and this was partially noted in \cite{PS25}. 
    \item The involutions $w_2$ and $w_{26}$ on the curve $X_0^{39}(2)$ are swapped. In this case the respective quotients are genus $4$, while the quotients by the subgroups $\langle w_2, w_3 \rangle$ and $\langle w_3, w_{26} \rangle$ have genus $2$. We noted this error while checking our algorithms from \cref{Kodaira_section} against genus $2$ curves contained in \cite{GY17}, but we mention it here in order to collect these items in a single remark. 
     \item We also notice that there is an error in the equations for Atkin--Lehner involutions in \cite{GY17} for the curve $X_0^{15}(4)$. It seems likely that the equations for the involutions in the pair $\{w_{12}, w_{60}\}$ are swapped with the equations for those in $\{w_{4},w_{20}\}$. One can see there is an error from the fact that the genera of the involutions in the first pair should be $2$, while the genera of those in the second pair should be $3$; these genera are swapped according to the models given in \cite{GY17}. It is not clear to us which involution is swapped for which; point counts are not helpful as the quotients by the two involutions in each pair have isogenous jacobians, and because $4$ is not squarefree our algorithmic tools from this paper are not of help. 
\end{itemize}
\end{remark}


\section{Non--elliptic genus one quotients}\label{non_elliptic_eqns_section}

\subsection{Equations of non-elliptic genus one quotients}

Let $N$ be squarefree and let $W \leq W_0(D,N)$ be a non-trivial Atkin--Lehner subgroup such that the quotient $X = X_0^D(N)/W$ is a non-elliptic curve of genus $1$. That is, $X$ has genus $1$ and $X(\Q) = \emptyset$ by the results of \cref{rational_points_subsection}. 

In this section, we seek to determine an equation for $X$. Our main strategy follows a method from in \cite{GR06}, in which the authors determine equations for non--elliptic genus $1$ Shimura curves $X_0^D(N)$ and $X_0^D(1)/\langle w_m \rangle$ with $w_m \in W_0(D,N)$ a non-trivial involution. We recall from their work our initial approach for determining candidate equations among twist families. We then state our main result (which recovers equations for the $17$ curves $X_0^D(1)/\langle w_m \rangle$ handled in \cite{GR06}). 

\begin{proposition}\cite[\S 2.1]{GR06}\label{proposition: GR06_twist_equations}
Let $X$ be a nice curve of genus $1$ over a field $F$ with $\text{char}(F) \neq 2$ satisfying $X(F) = \emptyset$. Suppose that we have the following information:
\begin{enumerate}
    \item an involution $w \in Aut(X)$ such that $X/\langle w \rangle \cong \P^1_F$, with corresponding covering map $\pi : X \to X/\langle w \rangle$,
    \item an element $d \in F^\times \setminus F^{\times 2}$ such that 
    \[ \pi(X(F(\sqrt{d}))) \cap \left(X/\langle w \rangle\right)\left(F\right) \neq \emptyset \; ,\]
    \item and a short Weierstrass equation $y^2 = x^3 + Ax+ B$ for the elliptic curve $E = \Jac(X)$, with $A,B \in F$. 
\end{enumerate}
Denote by $E_d$ the twist of $E$ by $d$:
\[ E_d : y^2 = x^3 + Ad^2x + Bd^3, \]
and let 
\[ \left\{ \infty, (a_1,b_1), \ldots, (a_r,b_r) \right\} \subseteq E_d(F) \]
be a complete, non-redundant set of representatives of the classes of the weak Mordell--Weil group $E_d(F)/2E_d(F)$. For $1 \leq i \leq r$, let 
\[ f_i(x) := x^4 - 6a_i x^2 + 8b_i x - 3a_i^2 - 4Ad^2\; , \quad 1 \leq i \leq r\; . \]
Then $C$ is isomorphic over $F$ to exactly one of the curves among those with equations $y^2 = d f_i(x)$ for $1 \leq i \leq r$. 
\end{proposition}

\begin{theorem}\label{theorem: non-elliptic_genus_one_equations}
Each of the curves $X = X_0^D(N)/W$ listed in \cref{table: non_elliptic_eqns} is non-elliptic of genus $1$. For the $4$ triples $(D,N,W)$ in the set 
    \begin{align*}
        \Big\{ &\left(14, 11, \langle w_2, w_{11} \rangle \right), \left(46, 3, \langle w_2, w_3 \rangle \right), \\  &\left(570, 1, \langle w_5, w_{57} \rangle \right), 
    \left(570, 1, \langle w_6, w_{190} \rangle\right) \Big\},
    \end{align*} 
    the curve $X$ admits exactly one of the two models given in \cref{table: non_elliptic_eqns} over $\Q$. Each of the remaining $146$ curves $X$ in \cref{table: non_elliptic_eqns} admits the given model over $\Q$. 
\end{theorem}
\begin{proof}
The fact that all of these genus $1$ curves lack rational points follows from the techniques mentioned in \cref{rational_points_subsection}. For each curve $X = X_0^D(N)/W$ in \cref{table: non_elliptic_eqns}, there exists a quadratic CM point $x$ on $X_0^D(N)$ which, using \cite[Corollary 5.14]{GR06}, we are able to show has rational image on a genus zero quotient of $X$ by an Atkin--Lehner involution $w_{m_0}$. Letting $K = \Q(x)  = \Q(\sqrt{d_0})$, with $d_0$ negative and squarefree, we have that $K$ is also necessarily the residue field of the image of $x$ on $X$ (as $X$ has no rational points). 

Therefore, we can take $d=d_0$ in our application of \cref{proposition: GR06_twist_equations} for $X$, whereas the required equation for $E = \Jac(X)$ is given by \cref{theorem: genus_one_jacobian_isomorphism_classes}. We compute in Magma representatives for $E_D(K)/2E_D(K)$, and for each non-trivial element we get a non-trivial twist $C$ of $X$. Let $\mathcal{C}_X$ denote the set of these twists. 

We next exclude from $\mathcal{C}_X$ any element $C$ which either
\begin{itemize}
\item disagrees with local points information for $X$ at finite primes $p \mid D$, computed using \cref{algorithm: local_points}, or
\item does not agree with the information of whether $X$ has points over $\R$, determined using Ogg's result \cite[Proposition 1]{Ogg83}. 
\end{itemize}
Code for all of these checks can be found in the file
\[ \text{\texttt{non{\_}elliptic{\_}genus{\_}1{\_}eqn{\_}computations.m}} \] 
in \cite{PSRepo}. Using just this information, we narrow $\mathcal{C}_X$ down to either a single (necessarily correct) model or two candidate models. 

In the latter case, which occurs for $53$ curves, we attempt one final check, similar to the main check used in 
\cite{GR06}, on the two remaining candidate equations. For this, fix $W' \leq W_0(D,N)$ with $[W' : W] = 2$ such that $X' = X_0^D(N)/W'$ is isomorphic to $\P^1_\Q$ and has a rational CM point determined as discussed above, such that the cover $\pi : X \to X'$ satisfies the required hypotheses used in our application of \cref{proposition: GR06_twist_equations} to obtain our candidate equations. Let $R \subseteq X(\overline{\Q})$ denote the set of ramification points of $\pi$. If $y^2 = f(x)$ is a hyperelliptic model for $X$ obtained via \cref{proposition: GR06_twist_equations}, then the splitting field of $f$ must be isomorphic to the field $L_\pi$ defined as the compositum of the residue fields of all points in $R$. 

Working in our level of generality, we do not always have complete information of the residue fields of points in $R$; the result \cite[Corollary 5.13]{GR06} proven and used by Gonz\'{a}lez--Rotger is specific to quotients of $X_0^D(N)$ by a single Atkin--Lehner involution. We know the following, though: the elements of $R$ are exactly the images under the natural map of fixed points on $X_0^D(N)$ of elements of the set $W' \setminus W$. Each such fixed point is given by \cref{prop: Ogg_fixed_pts} as a point with CM by an order $\mathfrak{o}$ in an imaginary quadratic field $K$, whose residue field on $X_0^D(N)$ over $K$ is the ring class field corresponding to $\mathfrak{o}$ \cite[Main Theorem II]{Sh67}. Let $\widetilde{R}$ denote the set of points on $X_0^D(N)$ fixed by some element of $W' \setminus W$, and let $\widetilde{L}_\pi$ denote the compositum of all ring class fields of the imaginary quadratic orders corresponding to points in $\widetilde{R}$.

It follows from above that $L_\pi$ is a subfield of $\widetilde{L_\pi}$. Therefore, if $y^2 = f_1(x)$ and $y^2 = f_2(x)$ are our two candidate equations for $X$ and the splitting field of $f_i$
is not contained in $\widetilde{L_\pi}$ for an index $i \in \{1,2\}$, we can discard the corresponding model from $\mathcal{C}_X$. This final check succeeds for all remaining curves aside from the six in the set included in the statement of this theorem. For the first two curves in this list, both $f_1$ and $f_2$ have splitting field contained in $\widetilde{L_\pi}$; this final check is too coarse to succeed. For the remaining four curves, we were not able to complete this required subfield check in Magma. Code for this final strategy can be found in the file \texttt{non{\_}elliptic{\_}fixed{\_}pt{\_}fld{\_}checks.m} in \cite{PSRepo}.
\end{proof}

\begin{remark}
The checks regarding the splitting field of the hyperelliptic equation and the field $L_\pi$ generated by the ramifcation points of the hyperelliptic quotient map $\pi$ described in the proof of \cref{theorem: non-elliptic_genus_one_equations} is one of the more computationally demanding components of this work. For the curves to which we apply this strategy, these checks took roughly $70$ hours of total computational time on our machine. 
\end{remark}

\begin{example}
Consider the genus $1$ curve $X = X_0^{21}(5)/\langle w_3, w_5 \rangle$. This curve has Kodaira symbols $I_1$ and $I_2$ the the primes $3$ and $7$, respectively. It lacks $\Q_7$-points, and is therefore not an elliptic curve. On the other hand, it has a $\Q_3$-point and is seen to have a real point due to its cover $X_0^{21}(5)/\langle w_5 \rangle$ having a real point by \cite[Proposition 1]{Ogg83}. 

The genus $0$ quotient $X^* = X_0^{21}(5)/\langle w_3, w_5, w_7 \rangle$ of $X$ is isomorphic to $\mathbb{P}^1_\Q$, given that it has a rational $-4$-CM point. We know from \cref{theorem: genus_one_jacobian_isomorphism_classes} that $E = \Jac(X)$ is isomorphic to the elliptic curve with Cremona reference $21$a$5$, whose twist by $-4$ is isomorphic to 
\[ E_{-4} : y^2 = x^3 - 16257456x + 25230566016. \]
We compute in Magma that $E_{-4}(\Q) \cong \Z \times \Z/4\Z$, and that a set of representatives for non-identity elements of $E_{-4}(\Q)/2E_{-4}(\Q)$ is as follows:
\begin{align*} 
\Big\{ (2364, 3024), \left(\frac{58236}{25}, -\frac{15984}{125}\right), (3624, -117936) \Big\}
\end{align*}
Applying \cref{proposition: GR06_twist_equations}, we get one candidate hyperelliptic model for $X$ for each element of the above set:
\begin{align*}
    C_1 : y^2 &= -4x^4 + 56736x^2 - 96768x - 193057344\; , \\
    C_2 : y^2 &= -4x^4 + \frac{1397664}{25}x^2 + \frac{511488}{125}x - \frac{121877379648}{625} \; , \\
    C_3 : y^2 &= -4x^4 + 86976x^2 + 3773952x - 102518784 \; .
\end{align*}
The curve $C_2$ is locally solvable at $p=7$, and so cannot be isomorphic to $X$. Both $C_1$ and $C_3$ have real points, and so we consider the ramification points of the natural degree $2$ map $X \to X^*$. These are the images on $X$ of all fixed points on $X_0^{21}(5)$ of the involutions $w_7, w_{21}, w_{35}, w_{105}$. Only $w_{105}$ has fixed points, namely the single $-420$-CM point on $X_0^{21}(5)$. Letting $K = \Q(\sqrt{-105})$, the residue field of this point over $\Q$ is a totally complex index $2$ subfield of the ring class field $H$ corresponding to the maximal order in $K$, and its image on $X_0^{21}(5)/\langle w_3 \rangle$ has degree $4$. The splitting field of the defining polynomial in $x$ for $C_1$ is not contained in this residue field (it is not even contained in $H$), hence $X$ is isomorphic over $\Q$ to $C_3$. 
\end{example}

\subsection{A return to rational points}\label{non_elliptic_rat_pts_section}

We now briefly return to the investigation of the existence of rational points on genus $1$ quotients started in \cref{rational_points_subsection}. We propose a final strategy, using the material of the prior section, to prove a genus $1$ Atkin--Lehner quotient is an elliptic curve with no explicit knowledge of its rational points.

Let $X = X_0^D(N)/W$ be a genus $1$ Atkin--Lehner quotient, let $w \in \Aut(X)$ be an Atkin--Lehner involution, and let $x \in X_0^D(N)$ be a CM point whose image on $X/\langle w \rangle$ is rational. Supposing that $X$ itself has no rational points, we can apply \cref{proposition: GR06_twist_equations} to compute a set $\mathcal{C}_X$ of candidate models for $X$. We can then discard candidates using the same methods used in the proof of \cref{theorem: non-elliptic_genus_one_equations}, by testing the existence of points over $\Q_p$ for $p \mid D$, the existence of real points, and the agreement of our model with the residue fields of fixed points of $w$. If these methods leave $\mathcal{C}_W$ empty, then we have proved by way of contradiction that $X$ must have a rational point and be an elliptic curve over $\Q$. This method succeeds for $28$ genus $1$ curves for which we remained unsure of the existence of rational points on $X$ following the methods of \cref{rational_points_subsection}, and these curves are listed in the file 
\[ \texttt{genus{\_}1{\_}AL{\_}quotients{\_}rat{\_}pts{\_}by{\_}non{\_}elliptic{\_}test{\_}final.m} \] 
in \cite{PSRepo}. 

In attempting the above method, we reduce to exactly one candidate non-elliptic equation, which we are not able to eliminate by our techniques, for $25$ curves. This leads to the following proposition.

\begin{proposition}
For each of the $25$ genus $1$ curve $X = X_0^D(N)/W$ with $(D,N,W)$ listed in \cref{table: unknown_possible_non_elliptic_eqns}, we have that either $X$ has a rational point (and is thus isomorphic to its Jacobian, whose isomorphism class is give in \cref{table: genus_1_sqfree_iso_classes}), or $X(\Q) = \emptyset$ and $X$ admits the model $y^2 = f(x)$ over $\Q$ for the given polynomial $f$. 
\end{proposition}
\begin{proof}
For each triple $(D,N,W)$ appearing in this table, we are able to find a CM point on $X_0^D(N)$ with the necessary setup to apply \cref{proposition: GR06_twist_equations}. We then generate a list of candidate equations, one of which must be an equation for $X_0^D(N)/W$ if this curve lacks a rational point. Using the techniques described in the proof of \cref{theorem: non-elliptic_genus_one_equations}, we are able to exclude all but one candidate model, and hence $X_0^D(N)/W$ is either an elliptic curve or admits this model. 
\end{proof}

{\begin{center}
\rowcolors{2}{white}{gray!20}
\begin{longtable}{|c|c|c|c|}
\caption{Possible non-elliptic models of the form $y^2=f(x)$ for   $25$ genus $1$ curves $X = X_0^D(N)/W$ for which we remain unsure of whether $X$ has a rational point. For each $X$, the hyperelliptic involution with respect to the given model (if correct) is the image of the Fricke involution $w_{DN}$ on $X$.}\label{table: unknown_possible_non_elliptic_eqns} \\
\hline
\rowcolor{gray!50}
$D$ & $N$ & $W$ & $f$ \\ \hline 
$ 6 $  &  $ 35 $  &  $ \langle
w_{ 10 }
,
w_{ 42 }
\rangle $ & $ -19x^4 - 1003770x^2 - 3480192x - 13271278707
$ \\ \hline
$ 6 $  &  $ 85 $  &  $ \langle
w_{ 2 }
,
w_{ 15 }
,
w_{ 51 }
\rangle $ & $ -4x^4 + 1733355/8x^2 - 434388339/16x + 974320350867/1024
$ \\ \hline
$ 6 $  &  $ 115 $  &  $ \langle
w_{ 5 }
,
w_{ 6 }
,
w_{ 46 }
\rangle $ & $ -19x^4 + 123462x^2 + 71114112x + 5986363725
$ \\ \hline
$ 6 $  &  $ 161 $  &  $ \langle
w_{ 2 }
,
w_{ 21 }
,
w_{ 69 }
\rangle $ & $ -19x^4 + 1059174x^2 - 17778528x - 16947937395
$ \\ \hline
$ 10 $  &  $ 21 $  &  $ \langle
w_{ 5 }
,
w_{ 21 }
\rangle $ & $ -3x^4 - 25434x^2 - 51667875
$ \\ \hline
$ 10 $  &  $ 33 $  &  $ \langle
w_{ 2 }
,
w_{ 55 }
\rangle $ & $ -8x^4 + 37296x^2 - 1270080x + 27705240
$ \\ \hline
$ 10 $  &  $ 39 $  &  $ \langle
w_{ 3 }
,
w_{ 5 }
,
w_{ 13 }
\rangle $ & $ -3x^4 + 34182x^2 + 139968x - 86869827
$ \\ \hline
$ 10 $  &  $ 51 $  &  $ \langle
w_{ 2 }
,
w_{ 15 }
,
w_{ 51 }
\rangle $ & $ -8x^4 + 116460x^2 - 7862184x + 303684147/2
$ \\ \hline
$ 10 $  &  $ 61 $  &  $ \langle
w_{ 10 }
,
w_{ 122 }
\rangle $ & $ -3x^4 + 486x^2 - 233280x - 4918563
$ \\ \hline
$ 10 $  &  $ 93 $  &  $ \langle
w_{ 3 }
,
w_{ 10 }
,
w_{ 62 }
\rangle $ & $ -3x^4 + 30294x^2 - 1166400x - 2994003
$ \\ \hline
$ 14 $  &  $ 57 $  &  $ \langle
w_{ 2 }
,
w_{ 21 }
,
w_{ 57 }
\rangle $ & $ -8x^4 - 27828x^2 - 327240x + 31947939/2
$ \\ \hline
$ 15 $  &  $ 14 $  &  $ \langle
w_{ 2 }
,
w_{ 5 }
,
w_{ 7 }
\rangle $ & $ -7x^4 + 72324x^2 - 177811200
$ \\ \hline
$ 21 $  &  $ 26 $  &  $ \langle
w_{ 2 }
,
w_{ 21 }
,
w_{ 39 }
\rangle $ & $ -4x^4 + 60768x^2 + 3483648x + 40927680
$ \\ \hline
$ 34 $  &  $ 5 $  &  $ \langle
w_{ 10 }
,
w_{ 34 }
\rangle $ & $ -11x^4 + 76230x^2 - 2299968x + 26629317
$ \\ \hline
$ 34 $  &  $ 7 $  &  $ \langle
w_{ 2 }
,
w_{ 17 }
\rangle $ & $ -3x^4 - 4698x^2 + 18144x - 2002563
$ \\ \hline
$ 74 $  &  $ 5 $  &  $ \langle
w_{ 10 }
,
w_{ 74 }
\rangle $ & $ -19x^4 - 32490x^2 - 11852352x + 448352253
$ \\ \hline
$ 119 $  &  $ 2 $  &  $ \langle
w_{ 7 }
,
w_{ 17 }
\rangle $ & $ -7x^4 + 122094x^2 + 7547904x + 108672921
$ \\ \hline
$ 210 $  &  $ 1 $  &  $ \langle
w_{ 7 }
,
w_{ 15 }
\rangle $ & $ -43x^4 - 7695882x^2 - 4070670336x - 492140268747
$ \\ \hline
$ 210 $  &  $ 19 $  &  $ \langle
w_{ 6 }
,
w_{ 7 }
,
w_{ 10 }
,
w_{ 19 }
\rangle $ & $ -67x^4 - 23674986x^2 - 7795776960x - 290611947987
$ \\ \hline
$ 330 $  &  $ 1 $  &  $ \langle
w_{ 2 }
,
w_{ 33 }
\rangle $ & $ -3x^4 + 1529604x^2 - 191556845568
$ \\ \hline
$ 330 $  &  $ 1 $  &  $ \langle
w_{ 3 }
,
w_{ 10 }
\rangle $ & $ -3x^4 - 23004x^2 + 83980800
$ \\ \hline
$ 462 $  &  $ 1 $  &  $ \langle
w_{ 11 }
,
w_{ 14 }
\rangle $ & $ -4x^4 - 605376x^2 - 22889889792
$ \\ \hline
$ 798 $  &  $ 1 $  &  $ \langle
w_{ 2 }
,
w_{ 3 }
,
w_{ 19 }
\rangle $ & $ -4x^4 + 1784736/25x^2 + 27772416/125x - 186001386048/625
$ \\ \hline
$ 1230 $  &  $ 1 $  &  $ \langle
w_{ 3 }
,
w_{ 10 }
,
w_{ 82 }
\rangle $ & $ -3x^4 + 486x^2 - 17003520x - 153051363
$ \\ \hline
$ 1722 $  &  $ 1 $  &  $ \langle
w_{ 6 }
,
w_{ 14 }
,
w_{ 41 }
\rangle $ & $ -67x^4 + 98497638x^2 - 65484526464x + 12323344962141
$ \\ \hline
\end{longtable}
\end{center}}


\section{Genus two bielliptic quotients}\label{genus_2_bielliptics_section}

If $X$ is a bielliptic curve of genus $2$, then each bielliptic quotient has a rational point and is thus an elliptic curve over $\Q$ \cite[Corollary, p.45]{Kuhn}. Any genus $2$ bielliptic curve has exactly $2$ non-hyperelliptic (and thus necessarily bielliptic) involutions, which one can see from the possible automorphism groups in \cite[Theorem 2]{SV04}. Concretely: if $w$ is a bielliptic involution on $X$ and $\iota$ denote the hyperelliptic involution on $X$, then $\iota \circ w$ is the other bielliptic involution.

Among the $1580$ Atkin--Lehner quotients $X_0^D(N)/W$ of genus $2$, we enumerate those with at least one bielliptic involution which is an Atkin--Lehner involution. Note that if $X$ is bielliptic via an Atkin--Lehner involution $w$ and $[W_0(D,N) : W] \geq 4$, then both bielliptic involutions on $X$ are Atkin--Lehner. If $[W_0(D,N) : W] = 2$, though, then the hyperelliptic involution $\iota$ is non--Atkin--Lehner, as is the bielliptic involution $\iota \circ w$. The latter situation happens for each curve encountered in \cite{GR04}, as in each case here $N=1$ with $\omega(D) = 2$ and the authors study the genus $2$ bielliptic quotients $X_0^D(1)/\langle w_m \rangle$ with $m>1$.

\begin{example}
Consider the genus $2$ curve $X = X_0^6(133)/\langle w_7, w_{19}, w_{57} \rangle$. The only non-trivial Atkin--Lehner quotient $X_0^6(133)^*$ has genus $1$. Therefore, the hyperelliptic involution $\iota$ and the bielliptic involution $\iota \circ w_2$ are both non-Atkin--Lehner.
\end{example}

We compute directly that there are exactly $388$ bielliptic genus $2$ curves $X_0^D(N)/W$ with two bielliptic Atkin--Lehner quotients, and there are exactly $251$ such curves with exactly one bielliptic Atkin--Lehner involution. We seek to generalize the results of González--Rotger \cite{GR04} to arbitrary level and quotients by computing equations for all of these curves. This includes recovering equations for the $10$ curves $X_0^D(1)/\langle w_m \rangle$ handled in \cite{GR04}. We use the following proposition from their work, which allows us to compute a finite set of candidate equations for a given bielliptic genus $2$ quotient $X$, to this aim. 

\begin{proposition}\cite[Proposition 2.1]{GR04}\label{GR_bielliptic_eqn_prop}
Let $C$ be a genus two curve defined over a field $k$ of characteristic not $2$ or $3$, and let $\iota$ be its hyperelliptic involution. Assume that $\Aut_k(C)$ contains a subgroup $\langle u_1, u_2 \rangle \cong (\Z/2\Z)^2$ with $u_2 = \iota \circ u_1$ and denote by $C_i$ the elliptic quotient $C/\langle u_i \rangle$ for $i \in \{1,2\}$. If
the two elliptic curves
\begin{align*}
    E_1 : Y^2 &= X^3 + A_1X + B_1 \\
    E_2 : Y^2 &= X^3 + A_2X + B_2
\end{align*}
are isomorphic over $k$ to $C_1$ and $C_2$, respectively, then there exists a solution $(a,b) \in k^\times \times k$ to the system of equations
\begin{align*}
    27a^3 B_2 &= 2A_1^3 + 27B_1^2 + 9A_1B_1b + 2A_1^2b^2 - B_1b^3 \\ 
    9a^2 A_2 &= -3A_1^2 + 9B_1b + A_1b^2.
\end{align*}
such that, letting 
\begin{align*}
    c &:= \frac{3A_1+b^2}{3a} \quad \quad \text{ and } \\
    d &:= \frac{27B_1 + 9A_1b + b^3}{27a^2} \; ,
\end{align*}
we have that $C$ admits the hyperelliptic equation $y^2 = ax^6 + bx^4 + cx^2 + d$ with $u_1$ given by $(x,y) \mapsto (-x,y)$.  
\end{proposition}

We now use \cref{GR_bielliptic_eqn_prop} to prove our main result. A notable distinction in the proof of our result, compared to that of \cite{GR04}, is that for the $10$ quotients they study the referenced proposition provides a single, necessarily correct, candidate model. In our case, we generally arrive at multiple candidate equations and we need to use arithmetic information to try to determine the correct model. 

\begin{theorem}\label{theorem: bielliptic_eqns}
\begin{enumerate}
    \item There are exactly $639$ genus $2$ curves $X = X_0^D(N)/W$, with $W \leq W_0(D,N)$ nontrivial, so that $X$ has a bielliptic Atkin--Lehner involution. 
    \item For the $405$ curves $X = X_0^D(N)/W$ listed in \cref{table: genus_2_bielliptic_eqns}, we have that $X$ is bielliptic of genus $2$ and admits the given equation over $\mathbb{Q}$. 
    \item If $X = X_0^D(N)/W$, with $W \leq W_0(D,N)$ nontrivial, is bielliptic of genus $2$ and does not have a bielliptic Atkin--Lehner involution, then $X$ must be among the $469$ quotients listed in \cref{table: bielliptic_non_AL_candidates}.
\end{enumerate}
\end{theorem}
\begin{proof}
We already described above the enumeration of all bielliptic genus $2$ Atkin--Lehner quotients with at least one bielliptic Atkin--Lehner involution, giving part (1), and so we proceed with part (2). Let $X = X_0^D(N)/W$ be such a quotient, and let $E_1$ be a bielliptic Atkin--Lehner quotient of $X$. If the other bielliptic involution on $X$ is Atkin--Lehner, then let $E_2$ be the other bielliptic quotient. Otherwise, let $E_2$ be any member of the isogeny class of the isogeny factor in $\Jac(X)$ not corresponding to $E_1,$ which we compute using \cref{Ribet_isog}. 

For $i \in \{1,2\}$, let $\mathcal{I}_i$ be the singleton set containing the isomorphism class of $E_i$ if $E_i$ is a bielliptic Atkin--Lehner quotient of $X$ and this isomorphism class was computed in \cref{genus_1_iso_classes_section}. Otherwise, let $\mathcal{I}_i$ denote the set of all members of the isogeny class of $E_i$ over $\Q$ (as computed in \cref{genus_1_isogeny_class_section} if $E_i$ is a bielliptic Atkin--Lehner quotient of $X$). 

For each pair $(E_1',E_2') \in \mathcal{I}_1 \times \mathcal{I}_2$, we attempt to compute all candidate equations for a bielliptic cover of $E_1'$ and $E_2'$ using \cref{GR_bielliptic_eqn_prop}. This works in all cases except for the three genus $2$ curves
\begin{align*} 
X_0^{390}(11)&/\langle w_2, w_5, w_{13}, w_{33} \rangle \; , \qquad X_0^{714}(5)/ \langle w_2, w_5, w_{17}, w_{21} \rangle, \\  &\text{ and } \qquad X_0^{798}(5)/\langle w_2, w_3, w_{19}, w_{35} \rangle \; .  
\end{align*}
For each of these $3$ curves, there is one non-Atkin--Lehner bielliptic quotient, and the Jacobian isogeny factor corresponding to this quotient has conductor properly dividing $DN$. We lack knowledge of the correct conductor, and the naive method using Ribet's isogeny with a decomposition of the Jacobian of the $D$-new part of $X_0(DN)$ seems infeasible to compute in a reasonable time due to the genus of $X_0(N)$ being rather large -- more specifically, the genera are $993, 849,$ and $945$. 

Aside from the above three curves, we succeed in computing candidate equations in each case and we let $\mathcal{H}_X$ denote the set of all such candidate equations over all pairs in $\mathcal{I}_1 \times \mathcal{I}_2$ up to isomorphism over $\Q$ of the hyperelliptic curve they define. Then, by design, $\mathcal{H}_X$ contains exactly one model for $X$.

If $\#\mathcal{H}_X = 1$, then we have determined the equation for $X$. Otherwise, we attempt to exclude elements of $\mathcal{H}_X$ as possible candidate equations using two strategies. If we know that $X$ has real points using Ogg's criterion then we can exclude any models which lack real points. Our main tool, though, is as follows: we compute the number of edges in the dual graph of the minimal regular model of $X$ over $\Z_p$ at all odd primes $p \mid D$ using the same method of \cref{algorithm: Kodaira_alg},\footnote{Of course, this is coarser information than the reduction types themselves. In our case, we check that it is enough to distinguish between any candidates for which there is at least one odd prime $p \mid D$ at which their reduction types differ except for the curve $X_0^{14}(17)/\langle w_7, w_{34} \rangle$. In this case, it helps us exclude one candidate equation, but we are still left with two with the same reduction types at $p=7$ and $p=17$. Note that we ignore reductions at $2$ when $2 \mid DN$ only because computations at $2$ are not implemented in the referenced package; we are able to compute dual graphs at $2$ for Atkin--Lehner quotients.} and we compare these counts to those of each hyperelliptic curve $H$ with a model in $\mathcal{H}_X$. This is done in Magma using Tim Dokchitser's package \texttt{redlib} \cite{redlib}, which is based on the works \cite{Dokchitser_models, Muselli} as well as \cite{NU73} in the genus $2$ case. If there is at least one prime $p \mid D$ at which there is a discrepancy, then we discard the model for $H$ from $\mathcal{H}_X$. The code for these computations are contained in the files \texttt{genus{\_}2{\_}bielliptic{\_}eqns.m} and \texttt{narrowing{\_}bielliptic{\_}eqns.m} in \cite{PSRepo}. 

Once more, if we are left with a single candidate equation in $\mathcal{H}_X$ then that is the equation for $X$ and we list the curve in \cref{table: genus_2_bielliptic_eqns} as stated in part (2).

For the $941$ genus $2$ Atkin--Lehner quotients which are not bielliptic via an Atkin--Lehner involution, we prove that $472$ are not bielliptic by computing that $\Jac(X)$ is simple over $\Q$. The relevant code for these computations is included in the file 
\[ \text{\texttt{computing{\_}genus{\_}2{\_}bielliptics.m}} \] 
in \cite{PSRepo}, and the $469$ genus $2$ quotients which are not handled by this check (as referenced in the statement of part (3) of this theorem) are listed in the file
\[ \text{\texttt{genus{\_}2{\_}remaining{\_}non{\_}AL{\_}bielliptic{\_}candidates.m}} \] 
in \cite{PSRepo} as well as in \cref{table: bielliptic_non_AL_candidates} as stated in part (3). 
\end{proof}

\begin{example}
Consider the genus $2$ curve $X := X_0^{51}(2)/\langle w_{51} \rangle$, with bielliptic Atkin--Lehner quotients $E_1 := X/\langle w_3 \rangle$ and $E_2 := X/\langle w_2 \rangle$. From our work in \cref{genus_1_iso_classes_section}, we know that $E_1$ is isomorphic over $\Q$ to the elliptic curve with Cremona label 102b1 and that $E_2$ is isomorphic to that with Cremona label 102a1. Using \cref{GR_bielliptic_eqn_prop}, we find that $X$ is isomorphic over $\Q$ to one of the following two non-isomorphic candidate hyperelliptic curves (each coming from a rational point on the zero-dimensional scheme with equations in terms of the coefficients of short Weierstrass equations for $E_1$ and $E_2$ as given in the referenced proposition):
\begin{align*}
    H_1 : y^2 &= -216x^6 + 225x^4 + 126x^2 + 9 \\
    H_2 : y^2 &= \frac{6912}{17}x^6 + \frac{585}{17}x^4 - \frac{1827}{17}x^2
        + \frac{288}{17}\; .
\end{align*}
At $p=17$, we have the following graphs $G_W$ corresponding to each $W \leq W_0(51,2)$ in $\{\{\textnormal{Id}\}, \langle w_{51} \rangle, \langle w_3, w_{17} \rangle, \langle w_2, w_{51} \rangle\}$. From each $G_W$, we can minimize and resolve to get the dual graph of the minimal regular model of $X_0^D(N)/W$ over $\Z_p$:
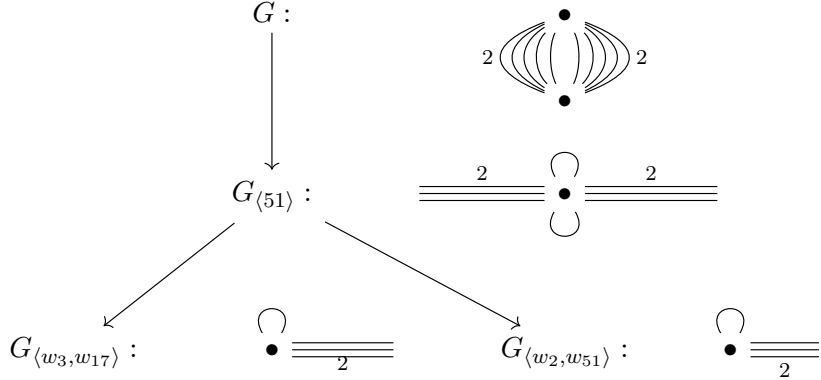
\[\begin{tikzcd}
	& {G:} && \bullet \\
	&&& \bullet \\
	& {G_{\langle 51 \rangle}:} & {} & \bullet & {} \\
	\\
	{G_{\langle w_3, w_{17} \rangle}:} & \bullet & {} & {G_{\langle w_2, w_{51} \rangle}:} & \bullet & {}
	\arrow[from=1-2, to=3-2]
	\arrow["2", curve={height=-30pt}, no head, from=1-4, to=2-4]
	\arrow["2"', curve={height=30pt}, no head, from=1-4, to=2-4]
	\arrow[curve={height=-24pt}, no head, from=1-4, to=2-4]
	\arrow[curve={height=-18pt}, no head, from=1-4, to=2-4]
	\arrow[curve={height=-12pt}, no head, from=1-4, to=2-4]
	\arrow[curve={height=-6pt}, no head, from=1-4, to=2-4]
	\arrow[curve={height=24pt}, no head, from=1-4, to=2-4]
	\arrow[curve={height=18pt}, no head, from=1-4, to=2-4]
	\arrow[curve={height=12pt}, no head, from=1-4, to=2-4]
	\arrow[curve={height=6pt}, no head, from=1-4, to=2-4]
	\arrow[from=3-2, to=5-1]
	\arrow[from=3-2, to=5-4]
	\arrow[no head, from=3-4, to=3-3]
	\arrow[shift left, no head, from=3-4, to=3-3]
	\arrow["2"', shift right, no head, from=3-4, to=3-3]
	\arrow[no head, from=3-4, to=3-4, loop, in=60, out=120, distance=5mm]
	\arrow[no head, from=3-4, to=3-4, loop, in=240, out=300, distance=5mm]
	\arrow[no head, from=3-4, to=3-5]
	\arrow[shift right, no head, from=3-4, to=3-5]
	\arrow["2", shift left, no head, from=3-4, to=3-5]
	\arrow[no head, from=5-2, to=5-2, loop, in=60, out=120, distance=5mm]
	\arrow[shift left, no head, from=5-2, to=5-3]
	\arrow[shift right, no head, from=5-2, to=5-3]
	\arrow["2"', no head, from=5-2, to=5-3]
	\arrow[no head, from=5-5, to=5-5, loop, in=60, out=120, distance=5mm]
	\arrow[shift left, no head, from=5-5, to=5-6]
	\arrow["2"', shift right, no head, from=5-5, to=5-6]
	\arrow[no head, from=5-5, to=5-6]
\end{tikzcd}\]
Using the library \cite{redlib}, we find that $H_1$ has reduction type $I_{1,1,0}$ while $H_2$ has reduction type $I_{1}-I_{1}-1$ using the nomenclature of Namikawa--Ueno's classification \cite[p. 179]{NU73}. The former matches the minimization of $G_{\langle 51 \rangle}$, and so $X \cong H_1$. 
\end{example}

\begin{remark}
There are exactly $231$ bielliptic genus $2$ curves $X$ with at least one bielliptic Atkin--Lehner quotient for which we do not determine the isomorphism class in \cref{theorem: bielliptic_eqns} and which are not among the three curves mentioned in the proof of this theorem. For each of these, we do compute a finite list of candidate models for $X$. For $147$ such curves, we narrow the determination down to just two isomorphism classes. We do not list those curves and their candidate models in this paper for sake of brevity, but they can be found in the file \texttt{genus{\_}2{\_}bielliptics{\_}eqn{\_}not{\_}determined.m} in \cite{PSRepo}. 
\end{remark}

\begin{remark}\label{33_2_remark}
There is in fact one more strategy we attempted towards \cref{theorem: bielliptic_eqns} to exclude candidate equations. Let $X = X_0^D(N)/W_X$ and $Y = X_0^D(N)/W_Y$ be two bielliptic genus $2$ curves for which we have finite lists of candidate models and which share a common elliptic curve quotient. For a given candidate equation $h \in \mathcal{H}_X$ for $X$, we can let $H_X$ denote the corresponding hyperelliptic curve for $X$. Using the command \texttt{Degree2Subcovers} in Magma, we get models for the two bielliptic quotients of $H_X$. If there exists no candidate model $H_Y$ for $Y$ for which at least one of the bielliptic quotients of $H_X$ is isomorphic to a bielliptic quotient of $H_Y$, then $H_X \not \cong X$ and $h$ can be discarded. 

While this strategy helps narrow down some candidates, it does not uniquely determine the isomorphism class of any remaining curves. However, a related argument lets us determine the isomorphism class of the Jacobian of the only genus $1$ quotient with $N$ squarefree we did not already handle in \cref{genus_1_iso_classes_section}, namely the curve $Z = X_0^{33}(2)/\langle w_6, w_{22} \rangle$. The genus $2$ curves $X = X_0^{33}(2)/ \langle w_{22} \rangle$ and $Y = X_0^{33}(2)/ \langle w_{33} \rangle$ are both degree $2$ covers of $Z$, and we have two candidate models each for $X$ and $Y$ following the strategies of the proof of \cref{theorem: bielliptic_eqns}. We know from our prior work that $Z$ has Cremona reference either $66$b$1$ or $66$b$3$, and $66$b$1$ is the only reference among this pair that appears among those of the bielliptic quotients of the candidate models for $Y$. We then know the isomorphism class of $Z$, though $66$b$1$ appears as a bielliptic quotient of both candidate models of $X$ and of $Z$ and therefore we don't learn the isomorphism class of either of these genus $2$ curves. 
\end{remark}

\begin{remark}
One would of course like to obtain a stronger version of part (3) of \cref{theorem: bielliptic_eqns}, determining which, if any, of the referenced Atkin--Lehner quotients are bielliptic by a non-Atkin--Lehner bielliptic involution. This seems to be a significant problem on its own, requiring further tools, and so we do not approach it further in this work. One may, for example, hope to use results on Atkin--Lehner quotients having no non-Atkin--Lehner involutions to extinguish this possibility for certain pairs, but the only result we are aware of on this (\cite[Theorem 1.3]{MPSS25}) unfortunately does not apply to \emph{any} of the remaining genus $2$ candidates from part (3) of \cref{theorem: bielliptic_eqns}. 
\end{remark}


\section{Tables}


{\begin{center}
\rowcolors{2}{white}{gray!20}
\begin{longtable}{|c|c|c|c||c|c|c|c|}
\caption{All $1270$ genus $1$ curves $X = X_0^D(N)/W$ with $W \leq W_0(D,N)$ nontrivial and $N$ squarefree. For each $X$, we give the isomorphism class of the elliptic curve $\Jac(X)$ via its Cremona label, ref.}\label{table: genus_1_sqfree_iso_classes} \\
\hline
\rowcolor{gray!50}
$D$ & $N$ & $W$ & ref & $D$ & $N$ & $W$ & ref \\ \hline 
$ 6 $  &  $ 5 $  &  $ \langle
w_{ 3 }
\rangle $ &  30a8
 & 
$ 6 $  &  $ 5 $  &  $ \langle
w_{ 5 }
\rangle $ &  30a7
\\ \hline
$ 6 $  &  $ 5 $  &  $ \langle
w_{ 15 }
\rangle $ &  30a3
 & 
$ 6 $  &  $ 5 $  &  $ \langle
w_{ 3 }
,
w_{ 5 }
\rangle $ &  30a6
\\ \hline
$ 6 $  &  $ 7 $  &  $ \langle
w_{ 2 }
\rangle $ &  42a5
 & 
$ 6 $  &  $ 7 $  &  $ \langle
w_{ 7 }
\rangle $ &  42a2
\\ \hline
$ 6 $  &  $ 7 $  &  $ \langle
w_{ 14 }
\rangle $ &  42a6
 & 
$ 6 $  &  $ 7 $  &  $ \langle
w_{ 2 }
,
w_{ 7 }
\rangle $ &  42a3
\\ \hline
$ 6 $  &  $ 11 $  &  $ \langle
w_{ 6 }
\rangle $ &  66c3
 & 
$ 6 $  &  $ 11 $  &  $ \langle
w_{ 22 }
\rangle $ &  66b2
\\ \hline
$ 6 $  &  $ 11 $  &  $ \langle
w_{ 33 }
\rangle $ &  66a3
 & 
$ 6 $  &  $ 11 $  &  $ \langle
w_{ 2 }
,
w_{ 3 }
\rangle $ &  66c4
\\
\hline
\end{longtable}
\end{center}}


{\begin{center}
\rowcolors{2}{white}{gray!20}
\begin{longtable}{|c|c|c|c||c|c|c|c|}
\caption{All $82$ genus $1$ curves $X = X_0^D(N)/W$ with $W \leq W_0(D,N)$ nontrivial and $N$ not squarefree. For each $X$, we give the isogeny class of the elliptic curve $\Jac(X)$ via its Cremona reference, ref.}\label{table: genus_1_not_sqfree_isog_classes} \\
\hline
\rowcolor{gray!50}
$D$ & $N$ & $W$ & ref & $D$ & $N$ & $W$ & ref \\ \hline 
$ 6 $  &  $ 25 $  &  $ \langle
w_{ 150 }
\rangle $ &  30a
 & 
$ 6 $  &  $ 25 $  &  $ \langle
w_{ 2 }
,
w_{ 3 }
\rangle $ &  150a
\\ \hline
$ 6 $  &  $ 25 $  &  $ \langle
w_{ 2 }
,
w_{ 25 }
\rangle $ &  150c
 & 
$ 6 $  &  $ 25 $  &  $ \langle
w_{ 3 }
,
w_{ 25 }
\rangle $ &  30a
\\ \hline
$ 6 $  &  $ 25 $  &  $ \langle
w_{ 3 }
,
w_{ 50 }
\rangle $ &  30a
 & 
$ 6 $  &  $ 25 $  &  $ \langle
w_{ 6 }
,
w_{ 50 }
\rangle $ &  150b
\\ \hline
$ 6 $  &  $ 49 $  &  $ \langle
w_{ 2 }
,
w_{ 147 }
\rangle $ &  42a
 & 
$ 6 $  &  $ 49 $  &  $ \langle
w_{ 3 }
,
w_{ 49 }
\rangle $ &  294d
\\ \hline
$ 6 $  &  $ 49 $  &  $ \langle
w_{ 3 }
,
w_{ 98 }
\rangle $ &  294g
 & 
$ 6 $  &  $ 121 $  &  $ \langle
w_{ 2 }
,
w_{ 3 }
,
w_{ 121 }
\rangle $ &  66c
\\ \hline
$ 6 $  &  $ 125 $  &  $ \langle
w_{ 2 }
,
w_{ 3 }
,
w_{ 125 }
\rangle $ &  150a
 & 
$ 6 $  &  $ 169 $  &  $ \langle
w_{ 2 }
,
w_{ 3 }
,
w_{ 169 }
\rangle $ &  1014f
\\
\hline
\end{longtable}
\end{center}}


{\begin{center}
\rowcolors{2}{white}{gray!20}
\begin{longtable}{|c|c|c||c|c|c||c|c|c|}
\caption{$650$ genus $0$ curves $X = X_0^D(N)/W$, with $W \leq W_0(D,N)$ nontrivial, for which we prove $X(\Q) \neq \emptyset$.}\label{table: genus_0_rat_pts} \\
\hline
\rowcolor{gray!50}
$D$ & $N$ & $W$ & $D$ & $N$ & $W$ & $D$ & $N$ & $W$ \\ \hline 
$ 6 $  &  $ 1 $  &  $ \langle
w_{ 2 }
\rangle $
 & 
$ 6 $  &  $ 1 $  &  $ \langle
w_{ 3 }
\rangle $
 & 
$ 6 $  &  $ 1 $  &  $ \langle
w_{ 6 }
\rangle $
\\ \hline
$ 6 $  &  $ 1 $  &  $ \langle
w_{ 2 }
,
w_{ 3 }
\rangle $
 & 
$ 6 $  &  $ 5 $  &  $ \langle
w_{ 6 }
\rangle $
 & 
$ 6 $  &  $ 5 $  &  $ \langle
w_{ 30 }
\rangle $
\\ \hline
$ 6 $  &  $ 5 $  &  $ \langle
w_{ 2 }
,
w_{ 3 }
\rangle $
 & 
$ 6 $  &  $ 5 $  &  $ \langle
w_{ 2 }
,
w_{ 5 }
\rangle $
 & 
$ 6 $  &  $ 5 $  &  $ \langle
w_{ 2 }
,
w_{ 15 }
\rangle $
\\ \hline
$ 6 $  &  $ 5 $  &  $ \langle
w_{ 3 }
,
w_{ 10 }
\rangle $
 & 
$ 6 $  &  $ 5 $  &  $ \langle
w_{ 5 }
,
w_{ 6 }
\rangle $
 & 
$ 6 $  &  $ 5 $  &  $ \langle
w_{ 6 }
,
w_{ 10 }
\rangle $
\\ \hline
$ 6 $  &  $ 5 $  &  $ \langle
w_{ 2 }
,
w_{ 3 }
,
w_{ 5 }
\rangle $
 & 
$ 6 $  &  $ 7 $  &  $ \langle
w_{ 42 }
\rangle $
 & 
$ 6 $  &  $ 7 $  &  $ \langle
w_{ 2 }
,
w_{ 21 }
\rangle $
\\ \hline
$ 6 $  &  $ 7 $  &  $ \langle
w_{ 2 }
,
w_{ 3 }
\rangle $
 & 
$ 6 $  &  $ 7 $  &  $ \langle
w_{ 3 }
,
w_{ 7 }
\rangle $
 & 
$ 6 $  &  $ 7 $  &  $ \langle
w_{ 3 }
,
w_{ 14 }
\rangle $
\\
\hline

\end{longtable}
\end{center}}


{\begin{center}
\rowcolors{2}{white}{gray!20}
\begin{longtable}{|c|c|c||c|c|c||c|c|c|}
\caption{$537$ genus $1$ curves $X = X_0^D(N)/W$, with $W \leq W_0(D,N)$ nontrivial, for which we prove that $X$ is an elliptic curve ($X(\Q) \neq \emptyset$) and for which $X$ has positive rank over $\Q$.}\label{table: genus_1_rat_pts_pos_rank} \\
\hline
\rowcolor{gray!50}
$D$ & $N$ & $W$ & $D$ & $N$ & $W$ & $D$ & $N$ & $W$ \\ \hline 
$ 6 $  &  $ 17 $  &  $ \langle
w_{ 102 }
\rangle $
 & 
$ 6 $  &  $ 17 $  &  $ \langle
w_{ 6 }
,
w_{ 17 }
\rangle $
 & 
$ 6 $  &  $ 23 $  &  $ \langle
w_{ 138 }
\rangle $
\\ \hline
$ 6 $  &  $ 23 $  &  $ \langle
w_{ 6 }
,
w_{ 23 }
\rangle $
 & 
$ 6 $  &  $ 35 $  &  $ \langle
w_{ 6 }
,
w_{ 35 }
\rangle $
 & 
$ 6 $  &  $ 35 $  &  $ \langle
w_{ 5 }
,
w_{ 6 }
,
w_{ 7 }
\rangle $
\\ \hline
$ 6 $  &  $ 41 $  &  $ \langle
w_{ 246 }
\rangle $
 & 
$ 6 $  &  $ 41 $  &  $ \langle
w_{ 6 }
,
w_{ 41 }
\rangle $
 & 
$ 6 $  &  $ 43 $  &  $ \langle
w_{ 3 }
,
w_{ 86 }
\rangle $
\\ \hline
$ 6 $  &  $ 43 $  &  $ \langle
w_{ 6 }
,
w_{ 43 }
\rangle $
 & 
$ 6 $  &  $ 47 $  &  $ \langle
w_{ 2 }
,
w_{ 141 }
\rangle $
 & 
$ 6 $  &  $ 49 $  &  $ \langle
w_{ 3 }
,
w_{ 98 }
\rangle $
\\ \hline
$ 6 $  &  $ 53 $  &  $ \langle
w_{ 2 }
,
w_{ 159 }
\rangle $
 & 
$ 6 $  &  $ 53 $  &  $ \langle
w_{ 6 }
,
w_{ 53 }
\rangle $
 & 
$ 6 $  &  $ 55 $  &  $ \langle
w_{ 6 }
,
w_{ 10 }
,
w_{ 22 }
\rangle $
\\ \hline
$ 6 $  &  $ 59 $  &  $ \langle
w_{ 2 }
,
w_{ 177 }
\rangle $
 & 
$ 6 $  &  $ 59 $  &  $ \langle
w_{ 6 }
,
w_{ 59 }
\rangle $
 & 
$ 6 $  &  $ 61 $  &  $ \langle
w_{ 2 }
,
w_{ 183 }
\rangle $
\\
\hline
\end{longtable}
\end{center}}


{\begin{center}
\rowcolors{2}{white}{gray!20}
\begin{longtable}{|c|c|c||c|c|c||c|c|c|}
\caption{All $48$ genus $0$ curves $X = X_0^D(N)/W$, with $W \leq W_0(D,N)$ nontrivial, for which we remain unsure of whether $X(\Q) = \emptyset$.}\label{table: genus_0_unknown_rat_pts} \\
\hline
\rowcolor{gray!50}
$D$ & $N$ & $W$ & $D$ & $N$ & $W$ & $D$ & $N$ & $W$ \\ \hline 
$ 6 $  &  $ 25 $  &  $ \langle
w_{ 2 }
,
w_{ 75 }
\rangle $
 & 
$ 6 $  &  $ 25 $  &  $ \langle
w_{ 6 }
,
w_{ 25 }
\rangle $
 & 
$ 6 $  &  $ 35 $  &  $ \langle
w_{ 5 }
,
w_{ 6 }
,
w_{ 14 }
\rangle $
\\ \hline
$ 6 $  &  $ 35 $  &  $ \langle
w_{ 6 }
,
w_{ 7 }
,
w_{ 10 }
\rangle $
 & 
$ 6 $  &  $ 49 $  &  $ \langle
w_{ 6 }
,
w_{ 49 }
\rangle $
 & 
$ 6 $  &  $ 55 $  &  $ \langle
w_{ 3 }
,
w_{ 10 }
,
w_{ 22 }
\rangle $
\\ \hline
$ 6 $  &  $ 55 $  &  $ \langle
w_{ 6 }
,
w_{ 10 }
,
w_{ 11 }
\rangle $
 & 
$ 6 $  &  $ 77 $  &  $ \langle
w_{ 2 }
,
w_{ 21 }
,
w_{ 33 }
\rangle $
 & 
$ 6 $  &  $ 77 $  &  $ \langle
w_{ 6 }
,
w_{ 11 }
,
w_{ 14 }
\rangle $
\\ \hline
$ 10 $  &  $ 7 $  &  $ \langle
w_{ 2 }
,
w_{ 5 }
\rangle $
 & 
$ 10 $  &  $ 7 $  &  $ \langle
w_{ 5 }
,
w_{ 7 }
\rangle $
 & 
$ 10 $  &  $ 7 $  &  $ \langle
w_{ 10 }
,
w_{ 14 }
\rangle $
\\ \hline
$ 10 $  &  $ 9 $  &  $ \langle
w_{ 2 }
,
w_{ 45 }
\rangle $
 & 
$ 10 $  &  $ 9 $  &  $ \langle
w_{ 5 }
,
w_{ 18 }
\rangle $
 & 
$ 10 $  &  $ 21 $  &  $ \langle
w_{ 3 }
,
w_{ 5 }
,
w_{ 7 }
\rangle $
\\ \hline
$ 10 $  &  $ 21 $  &  $ \langle
w_{ 3 }
,
w_{ 10 }
,
w_{ 14 }
\rangle $
 & 
$ 10 $  &  $ 33 $  &  $ \langle
w_{ 2 }
,
w_{ 15 }
,
w_{ 33 }
\rangle $
 & 
$ 10 $  &  $ 69 $  &  $ \langle
w_{ 3 }
,
w_{ 5 }
,
w_{ 46 }
\rangle $
\\
\hline

\end{longtable}
\end{center}}


{\begin{center}
\rowcolors{2}{white}{gray!20}
\begin{longtable}{|c|c|c||c|c|c||c|c|c|}
\caption{All $54$ genus $1$ curves $X = X_0^D(N)/W$, with $W \leq W_0(D,N)$ nontrivial, for which we remain unsure of whether $X(\Q) = \emptyset$ and for which $\Jac(X)$ has positive rank over $\Q$.}\label{table: genus_1_unknown_rat_pts_pos_rank} \\
\hline
\rowcolor{gray!50}
$D$ & $N$ & $W$ & $D$ & $N$ & $W$ & $D$ & $N$ & $W$ \\ \hline 
$ 6 $  &  $ 155 $  &  $ \langle
w_{ 2 }
,
w_{ 5 }
,
w_{ 93 }
\rangle $
 & 
$ 6 $  &  $ 155 $  &  $ \langle
w_{ 5 }
,
w_{ 6 }
,
w_{ 31 }
\rangle $
 & 
$ 6 $  &  $ 169 $  &  $ \langle
w_{ 2 }
,
w_{ 3 }
,
w_{ 169 }
\rangle $
\\ \hline
$ 6 $  &  $ 203 $  &  $ \langle
w_{ 3 }
,
w_{ 7 }
,
w_{ 58 }
\rangle $
 & 
$ 6 $  &  $ 287 $  &  $ \langle
w_{ 2 }
,
w_{ 3 }
,
w_{ 7 }
,
w_{ 41 }
\rangle $
 & 
$ 6 $  &  $ 455 $  &  $ \langle
w_{ 2 }
,
w_{ 3 }
,
w_{ 5 }
,
w_{ 7 }
,
w_{ 13 }
\rangle $
\\ \hline
$ 10 $  &  $ 49 $  &  $ \langle
w_{ 5 }
,
w_{ 98 }
\rangle $
 & 
$ 10 $  &  $ 141 $  &  $ \langle
w_{ 2 }
,
w_{ 3 }
,
w_{ 5 }
,
w_{ 47 }
\rangle $
 & 
$ 10 $  &  $ 159 $  &  $ \langle
w_{ 2 }
,
w_{ 3 }
,
w_{ 5 }
,
w_{ 53 }
\rangle $
\\ \hline
$ 10 $  &  $ 161 $  &  $ \langle
w_{ 2 }
,
w_{ 5 }
,
w_{ 7 }
,
w_{ 23 }
\rangle $
 & 
$ 10 $  &  $ 231 $  &  $ \langle
w_{ 2 }
,
w_{ 3 }
,
w_{ 5 }
,
w_{ 7 }
,
w_{ 11 }
\rangle $
 & 
$ 14 $  &  $ 25 $  &  $ \langle
w_{ 14 }
,
w_{ 25 }
\rangle $
\\ \hline
$ 14 $  &  $ 95 $  &  $ \langle
w_{ 2 }
,
w_{ 5 }
,
w_{ 7 }
,
w_{ 19 }
\rangle $
 & 
$ 15 $  &  $ 26 $  &  $ \langle
w_{ 2 }
,
w_{ 13 }
,
w_{ 15 }
\rangle $
 & 
$ 15 $  &  $ 38 $  &  $ \langle
w_{ 2 }
,
w_{ 3 }
,
w_{ 5 }
,
w_{ 19 }
\rangle $
\\ \hline
$ 15 $  &  $ 82 $  &  $ \langle
w_{ 2 }
,
w_{ 3 }
,
w_{ 5 }
,
w_{ 41 }
\rangle $
 & 
$ 21 $  &  $ 55 $  &  $ \langle
w_{ 3 }
,
w_{ 5 }
,
w_{ 7 }
,
w_{ 11 }
\rangle $
 & 
$ 22 $  &  $ 35 $  &  $ \langle
w_{ 5 }
,
w_{ 7 }
,
w_{ 22 }
\rangle $
\\
\hline
\end{longtable}
\end{center}}


{\begin{center}
\rowcolors{2}{white}{gray!20}
\begin{longtable}{|c|c|c|c|}
\caption{Hyperelliptic models of the form $y^2=f(x)$ for $150$ non-elliptic genus $1$ curves $X = X_0^D(N)/W$. For each $X$, the hyperelliptic involution with respect to the given model(s) is the image of the Fricke involution $w_{DN}$ on $X$.}\label{table: non_elliptic_eqns} \\
\hline
\rowcolor{gray!50}
$D$ & $N$ & $W$ & $f$ \\ \hline 
$ 6 $  &  $ 5 $  &  $ \langle
w_{ 3 }
\rangle $ & 
$ -4x^4 - 75456x^2 - 31850496
$ \\ \hline
$ 6 $  &  $ 5 $  &  $ \langle
w_{ 5 }
\rangle $ & 
$ -4x^4 + 145728x^2 - 1327104000
$ \\ \hline
$ 6 $  &  $ 7 $  &  $ \langle
w_{ 2 }
\rangle $ & 
$ -3x^4 + 36612x^2 - 90699264
$ \\ \hline
$ 6 $  &  $ 7 $  &  $ \langle
w_{ 7 }
\rangle $ & 
$ -3x^4 - 10530x^2 - 8680203
$ \\ \hline
$ 6 $  &  $ 11 $  &  $ \langle
w_{ 6 }
\rangle $ & 
$ -19x^4 + 4507902x^2 + 3972672x - 268754525019
$ \\ \hline
$ 6 $  &  $ 11 $  &  $ \langle
w_{ 22 }
\rangle $ & 
$ -19x^4 + 193230x^2 + 1477440x - 619233579
$ \\
\hline
\end{longtable}
\end{center}}


{\begin{center}
\rowcolors{2}{white}{gray!20}
\begin{longtable}{|c|c|c|c|c|c|}
\caption{Equations for $405$ genus $2$ curves $X = X_0^D(N)/W$ which are bielliptic via an Atkin-Lehner involution. For each $X$, we give a single divisor $m_1$ of $DN$ so that $X/\langle w_{m_1} \rangle$ is a bielliptic quotient of $X$ if $X$ has a single bielliptic Atkin--Lehner quotient, and we similarly give divisors $m_1$ and $m_2$ if $X$ has two bielliptic Atkin--Lehner quotients. Our hyperelliptic model for $X$ is $y^2 = f(x)$ for the given polynomial $f$. }\label{table: genus_2_bielliptic_eqns} \\
\hline
\rowcolor{gray!50}
$D$ & $N$ & $W$ & $m_1$ & $m_2$ & $f$ \\ \hline 
$ 6 $  &  $ 11 $  &  $ \langle
w_{ 2 }
\rangle $ & 
$w_{ 3 }$ & $w_{ 11 }$ & 
$ -11979/2x^6 - 10917/2x^4 + 1035/2x^2 - 27/2
$ \\ \hline
$ 6 $  &  $ 11 $  &  $ \langle
w_{ 3 }
\rangle $ & 
$w_{ 2 }$ & $w_{ 11 }$ & 
$ 3267x^6 - 7092x^4 + 1143x^2 - 54
$ \\ \hline
$ 6 $  &  $ 11 $  &  $ \langle
w_{ 11 }
\rangle $ & 
$w_{ 2 }$ & $w_{ 3 }$ & 
$ 3267x^6 - 19044x^4 + 36999x^2 - 23958
$ \\ \hline
$ 6 $  &  $ 17 $  &  $ \langle
w_{ 17 }
\rangle $ & 
$w_{ 3 }$ & $w_{ 6 }$ & 
$ -235824x^6 - 42291x^4 - 2538x^2 - 51
$ \\ \hline
$ 6 $  &  $ 19 $  &  $ \langle
w_{ 2 }
\rangle $ & 
$w_{ 3 }$ & $w_{ 19 }$ & 
$ -171x^6 + 522x^4 + 2637x^2 - 8748
$ \\ \hline
$ 6 $  &  $ 19 $  &  $ \langle
w_{ 6 }
\rangle $ & 
$w_{ 2 }$ & $w_{ 38 }$ & 
$ 9747/4x^6 - 18945/4x^4 + 11385/4x^2 - 2187/4
$ \\
\hline

\end{longtable}
\end{center}}


{\begin{center}
\rowcolors{2}{white}{gray!20}
\begin{longtable}{|c|c|c||c|c|c||c|c|c|}
\caption{All $469$ genus $2$ curves $X = X_0^D(N)/W$ which are not bielliptic via an Atkin--Lehner involution for which we remain unsure of whether $X$ is geometrically bielliptic. }\label{table: bielliptic_non_AL_candidates} \\
\hline
\rowcolor{gray!50}
$D$ & $N$ & $W$ & $D$ & $N$ & $W$ & $D$ & $N$ & $W$ \\ \hline $ 6 $  &  $ 29 $  &  $ \langle
w_{ 3 }
,
w_{ 29 }
\rangle $
 & 
$ 6 $  &  $ 37 $  &  $ \langle
w_{ 6 }
,
w_{ 74 }
\rangle $
 & 
$ 6 $  &  $ 41 $  &  $ \langle
w_{ 2 }
,
w_{ 3 }
\rangle $
\\ \hline
$ 6 $  &  $ 41 $  &  $ \langle
w_{ 3 }
,
w_{ 41 }
\rangle $
 & 
$ 6 $  &  $ 43 $  &  $ \langle
w_{ 2 }
,
w_{ 3 }
\rangle $
 & 
$ 6 $  &  $ 43 $  &  $ \langle
w_{ 2 }
,
w_{ 43 }
\rangle $
\\ \hline
$ 6 $  &  $ 49 $  &  $ \langle
w_{ 2 }
,
w_{ 3 }
\rangle $
 & 
$ 6 $  &  $ 49 $  &  $ \langle
w_{ 2 }
,
w_{ 49 }
\rangle $
 & 
$ 6 $  &  $ 49 $  &  $ \langle
w_{ 6 }
,
w_{ 98 }
\rangle $
\\ \hline
$ 6 $  &  $ 55 $  &  $ \langle
w_{ 2 }
,
w_{ 5 }
,
w_{ 11 }
\rangle $
 & 
$ 6 $  &  $ 55 $  &  $ \langle
w_{ 3 }
,
w_{ 5 }
,
w_{ 11 }
\rangle $
 & 
$ 6 $  &  $ 59 $  &  $ \langle
w_{ 2 }
,
w_{ 59 }
\rangle $
\\ \hline
$ 6 $  &  $ 61 $  &  $ \langle
w_{ 2 }
,
w_{ 3 }
\rangle $
 & 
$ 6 $  &  $ 65 $  &  $ \langle
w_{ 3 }
,
w_{ 5 }
,
w_{ 13 }
\rangle $
 & 
$ 6 $  &  $ 73 $  &  $ \langle
w_{ 2 }
,
w_{ 3 }
\rangle $
\\ \hline
$ 6 $  &  $ 73 $  &  $ \langle
w_{ 3 }
,
w_{ 146 }
\rangle $
 & 
$ 6 $  &  $ 77 $  &  $ \langle
w_{ 2 }
,
w_{ 3 }
,
w_{ 7 }
\rangle $
 & 
$ 6 $  &  $ 77 $  &  $ \langle
w_{ 2 }
,
w_{ 7 }
,
w_{ 11 }
\rangle $
\\
\hline
\end{longtable}
\end{center}}


\section{Appendix: Dual graphs and Kodaira symbols in depth one}\label{Appendix: dual graphs}

Let $X = X_0^D(N)/\langle w_m \rangle$ be an Atkin--Lehner quotient of genus $1$ with $w_m \in W_0(D,N)$ a non-trivial involution. In \cref{Reg_models_table}, we draw the graph $G_\langle w_m\rangle$, as defined in \cref{dual_graphs_section}, associated to $X$ over $\Z_p$ for all primes $p \mid D$. 

From these graphs, the Kodaira symbols for $J = \Jac(X)$ at all such primes are determined by minimizing and resolving. The isomorphism class of $J$ is then determined just from this information and its isogeny class (see \cref{prop: genus_one_jacobian_isogeny_classes}), as described in the proof of \cref{theorem: genus_one_jacobian_isomorphism_classes}, for all but $5$ such curves $X$. For these $5$ curves, we list the two Cremona references for the isomorphism classes matching the Kodaira symbols of $J$ in blue, and we strikethrough the one which we prove is incorrect by other means as discussed in the proof of the referenced theorem. 

It is hoped that these explicitly worked examples will be useful to the reader seeking to gain comfort with {\v{C}}erednik--Drinfeld uniformizations, given that there are not plentiful such examples in the literature (especially in the level of generality of this work, allowing squarefree $N>1$ and general Atkin--Lehner subgroups $W$).

\begin{center}
\renewcommand{\arraystretch}{1.5}
\begin{longtable}{| c | c | c | c | c |}\caption{Kodaira symbols at $p \mid D$ and the isomorphism class of $J = \Jac(X)$, via dual graphs of regular models over $\Z_p$, for all $129$ Atkin--Lehner quotients $X = X_0^D(N)/\langle w_m \rangle$ with genus $1$}\label{Reg_models_table} \\ \hline 
$(D,N,m)$ & $\substack{\text{Class} \\ \text{of } J}$ & $p$ & $G_{\langle w_m \rangle}$ & $\substack{\text{Symbol of} \\ J \text{ at } p}$ \\ \hline \hline

$(6,5,3)$  & 30a8 & $2$ & $\begin{tikzcd}
	\bullet & \bullet
	\arrow["2", curve={height=-6pt}, no head, from=1-1, to=1-2]
	\arrow[curve={height=6pt}, no head, from=1-1, to=1-2]
\end{tikzcd}$ &  $I_3$ \\ \cline{3-5}

  & & $3$ & $\begin{tikzcd}
	\bullet
	\arrow[no head, from=1-1, to=1-1, loop, in=55, out=125, distance=10mm]
\end{tikzcd}$ & $I_1$ \\ \hline

$(6,5,5)$ & 30a7 & $2$ & $\begin{tikzcd}
	\bullet & \bullet
	\arrow["2", curve={height=-6pt}, no head, from=1-1, to=1-2]
	\arrow[curve={height=6pt}, no head, from=1-1, to=1-2]
\end{tikzcd}$ &  $I_3$ \\ \cline{3-5} 

& & $3$ & $\begin{tikzcd}
	\bullet
	\arrow[no head, from=1-1, to=1-1, loop, in=55, out=125, distance=10mm]
\end{tikzcd}$ & $I_4$   \\ \hline

$(6,5,15)$ & 30a3 & $2$ & $\begin{tikzcd}
	\bullet & \bullet \\
	\bullet & \bullet
	\arrow["2", no head, from=1-1, to=1-2]
	\arrow["4"', no head, from=1-1, to=2-1]
	\arrow["4", no head, from=1-2, to=2-2]
	\arrow["2"', no head, from=2-1, to=2-2]
\end{tikzcd}$  & $I_{12}$ \\ \cline{3-5}

& & $3$ & $\begin{tikzcd}
	\bullet
	\arrow[no head, from=1-1, to=1-1, loop, in=55, out=125, distance=10mm]
\end{tikzcd}$  & $I_{1}$ \\
\hline
\end{longtable}
\end{center}

\bibliographystyle{amsalpha}
\bibliography{biblio}

\providecommand{\bysame}{\leavevmode\hbox to3em{\hrulefill}\thinspace}
\providecommand{\MR}{\relax\ifhmode\unskip\space\fi MR }
\providecommand{\MRhref}[2]{%
  \href{http://www.ams.org/mathscinet-getitem?mr=#1}{#2}
}
\providecommand{\href}[2]{#2}
\begin{thebibliography}{MPSS25}

\bibitem[Abr96]{Abramovich96}
Dan Abramovich, \emph{A linear lower bound on the gonality of modular curves},
  Internat. Math. Res. Notices (1996), no.~20, 1005--1011.

\bibitem[Bar99]{Bars99}
Francesc Bars, \emph{Bielliptic modular curves}, J. Number Theory \textbf{76}
  (1999), no.~1, 154--165.

\bibitem[BC91]{BC91}
J.-F. Boutot and H.~Carayol, \emph{Uniformisation {$p$}-adique des courbes de
  {S}himura}, no. 196-197, 1991, Courbes modulaires et courbes de Shimura
  (Orsay, 1987/1988), pp.~7, 45--158.

\bibitem[BCP97]{Magma}
W.~Bosma, J.~Cannon, and C.~Playoust, \emph{The {M}agma algebra system. {I}.
  {T}he user language ({M}agma {V}2.26-5)}, J. Symbolic Comput. \textbf{24}
  (1997), no.~3--4, 235--265, Computational algebra and number theory (London,
  1993).

\bibitem[BD96]{BD96}
M.~Bertolini and H.~Darmon, \emph{Heegner points on {M}umford-{T}ate curves},
  Invent. Math. \textbf{126} (1996), no.~3, 413--456.

\bibitem[Buz97]{Buzzard}
K.~Buzzard, \emph{Integral models of certain {S}himura curves}, Duke Math. J.
  \textbf{87} (1997), no.~3, 591--612.

\bibitem[{\v{C}}er76]{Cerednik}
I.~V. {\v{C}}erednik, \emph{Uniformization of algebraic curves by discrete
  arithmetic subgroups of {${\rm PGL}\sb{2}(k\sb{w})$} with compact quotient
  spaces}, Mat. Sb. (N.S.) \textbf{100(142)} (1976), no.~1, 59--88, 165.

\bibitem[DD15]{DD15}
Tim Dokchitser and Vladimir Dokchitser, \emph{Local invariants of isogenous
  elliptic curves}, Trans. Amer. Math. Soc. \textbf{367} (2015), no.~6,
  4339--4358.

\bibitem[DO24]{DO24}
Maarten Derickx and Petar Orli\'c, \emph{Modular curves {$X_0(N)$} with
  infinitely many quartic points}, Res. Number Theory \textbf{10} (2024),
  no.~2, Paper No. 42, 24.

\bibitem[Dok21]{Dokchitser_models}
Tim Dokchitser, \emph{Models of curves over discrete valuation rings}, Duke
  Math. J. \textbf{170} (2021), no.~11, 2519--2574.

\bibitem[Dok25]{redlib}
\bysame, \emph{redlib: {R}eduction types of curves in {M}agma, {P}ython and
  {J}avascript, version 3.4},
  \url{https://people.maths.bris.ac.uk/~matyd/redlib}, Accessed August 2025.

\bibitem[Dri76]{Drinfeld}
V.~G. Drinfeld, \emph{Coverings of {$p$}-adic symmetric domains}, Funkcional.
  Anal. i Prilo\v zen. \textbf{10} (1976), no.~2, 29--40.

\bibitem[GR04]{GR04}
Josep Gonz\'alez and Victor Rotger, \emph{Equations of {S}himura curves of
  genus two}, Int. Math. Res. Not. (2004), no.~14, 661--674.

\bibitem[GR06]{GR06}
Josep Gonz\'{a}lez and Victor Rotger, \emph{Non-elliptic {S}himura curves of
  genus one}, J. Math. Soc. Japan \textbf{58} (2006), no.~4, 927--948.

\bibitem[GY17]{GY17}
Jia-Wei Guo and Yifan Yang, \emph{Equations of hyperelliptic {S}himura curves},
  Compos. Math. \textbf{153} (2017), no.~1, 1--40.

\bibitem[HJ24]{HJ24}
Wontae Hwang and Daeyeol Jeon, \emph{Modular curves with infinitely many
  quartic points}, Math. Comp. \textbf{93} (2024), no.~345, 383--395.

\bibitem[HPS89a]{HPS89b}
Hiroaki Hijikata, Arnold~K. Pizer, and Thomas~R. Shemanske, \emph{The basis
  problem for modular forms on {$\Gamma_0(N)$}}, Mem. Amer. Math. Soc.
  \textbf{82} (1989), no.~418.

\bibitem[HPS89b]{HPS89a}
\bysame, \emph{Orders in quaternion algebras}, J. Reine Angew. Math.
  \textbf{394} (1989), 59--106.

\bibitem[Jeo21]{Jeon}
Daeyeol Jeon, \emph{Modular curves with infinitely many cubic points}, J.
  Number Theory \textbf{219} (2021), 344--355.

\bibitem[JL85]{JordanLivne}
Bruce~W. Jordan and Ron~A. Livn\'e, \emph{Local {D}iophantine properties of
  {S}himura curves}, Math. Ann. \textbf{270} (1985), no.~2, 235--248.

\bibitem[Kim03]{K03}
Henry~H. Kim, \emph{Functoriality for the exterior square of {${\rm GL}_4$} and
  the symmetric fourth of {${\rm GL}_2$}}, J. Amer. Math. Soc. \textbf{16}
  (2003), no.~1, 139--183, With appendix 1 by Dinakar Ramakrishnan and appendix
  2 by Kim and Peter Sarnak.

\bibitem[Kuh88]{Kuhn}
Robert~M. Kuhn, \emph{Curves of genus {$2$} with split {J}acobian}, Trans.
  Amer. Math. Soc. \textbf{307} (1988), no.~1, 41--49.

\bibitem[Kur79]{Kurihara}
Akira Kurihara, \emph{On some examples of equations defining {S}himura curves
  and the {M}umford uniformization}, J. Fac. Sci. Univ. Tokyo Sect. IA Math.
  \textbf{25} (1979), no.~3, 277--300.

\bibitem[Maz78]{MazurIsogenies}
B.~Mazur, \emph{Rational isogenies of prime degree (with an appendix by {D}.
  {G}oldfeld)}, Invent. Math. \textbf{44} (1978), no.~2, 129--162.

\bibitem[MPSS25]{MPSS25}
Pietro Mercuri, Oana Padurariu, Frederick Saia, and Claudio Stirpe,
  \emph{{Point counts, automorphisms, and gonalities of Shimura curves}},
  arXiv:2507.15992 (2025).

\bibitem[Mum72]{Mumford}
David Mumford, \emph{An analytic construction of degenerating curves over
  complete local rings}, Compositio Math. \textbf{24} (1972), 129--174.

\bibitem[Mus24]{Muselli}
Simone Muselli, \emph{Regular models of hyperelliptic curves}, Indag. Math.
  (N.S.) \textbf{35} (2024), no.~4, 646--697.

\bibitem[NR15]{NR15}
Joan Nualart~Riera, \emph{{On the Hyperbolic Uniformization of Shimura Curves
  with an Atkin-Lehner Quotient of Genus 0}}, Ph.D. thesis, 2015, Thesis
  (Ph.D.)--Universitat de Barcelona.

\bibitem[NU73]{NU73}
Yukihiko Namikawa and Kenji Ueno, \emph{The complete classification of fibres
  in pencils of curves of genus two}, Manuscripta Math. \textbf{9} (1973),
  143--186.

\bibitem[Ogg74]{Ogg74}
Andrew~P. Ogg, \emph{Hyperelliptic modular curves}, Bull. Soc. Math. France
  \textbf{102} (1974), 449--462.

\bibitem[Ogg83]{Ogg83}
\bysame, \emph{Real points on {S}himura curves}, Arithmetic and geometry,
  {V}ol. {I}, Progr. Math., vol.~35, Birkh\"{a}user Boston, Boston, MA, 1983,
  pp.~277--307.

\bibitem[Ogg85]{Ogg85}
\bysame, \emph{Mauvaise r\'{e}duction des courbes de {S}himura}, S\'{e}minaire
  de th\'{e}orie des nombres, {P}aris 1983--84, Progr. Math., vol.~59,
  Birkh\"{a}user Boston, Boston, MA, 1985, pp.~199--217.

\bibitem[PS23]{PS23}
Oana Padurariu and Ciaran Schembri, \emph{Rational points on {A}tkin--{L}ehner
  quotients of geometrically hyperelliptic {S}himura curves}, Expo. Math.
  \textbf{41} (2023), no.~3, 492--513.

\bibitem[PS25a]{PS25}
Oana Padurariu and Frederick Saia, \emph{Bielliptic {S}himura curves
  {$X_0^D(N)$} with nontrivial level}, Res. Number Theory \textbf{11} (2025),
  no.~1, Paper No. 2, 23.

\bibitem[PS25b]{PSRepo}
\bysame, \emph{{GenusAtMost2 Repository}},
  \url{https://github.com/fsaia/GenusAtMost2}, 2025.

\bibitem[Rib90]{Ribet90}
Kenneth~A. Ribet, \emph{On modular representations of {${\rm Gal}(\overline{\bf
  Q}/{\bf Q})$} arising from modular forms}, Invent. Math. \textbf{100} (1990),
  no.~2, 431--476.

\bibitem[Rot02]{Rotger02}
Victor Rotger, \emph{On the group of automorphisms of {S}himura curves and
  applications}, Compositio Math. \textbf{132} (2002), no.~2, 229--241.

\bibitem[Sai24]{Saia24}
Frederick Saia, \emph{C{M} points on {S}himura curves via {QM}-equivariant
  isogeny volcanoes}, Pacific J. Math. \textbf{332} (2024), no.~2, 321--384.

\bibitem[Shi67]{Sh67}
Goro Shimura, \emph{Construction of class fields and zeta functions of
  algebraic curves}, Ann. of Math. (2) \textbf{85} (1967), 58--159.

\bibitem[Shi75]{Sh75}
\bysame, \emph{On the real points of an arithmetic quotient of a bounded
  symmetric domain}, Math. Ann. \textbf{215} (1975), 135--164.

\bibitem[Sta14]{Stankewicz}
James Stankewicz, \emph{Twists of {S}himura curves}, Canad. J. Math.
  \textbf{66} (2014), no.~4, 924--960.

\bibitem[Sta25]{StankewiczCode}
\bysame, \emph{Code for computing dual graphs of special fibers of shimura
  curves}, \url{http://stankewicz.net/SpecialFiber.html}, Accessed August 2025.

\bibitem[SV04]{SV04}
Tanush Shaska and Helmut V\"olklein, \emph{Elliptic subfields and automorphisms
  of genus 2 function fields}, Algebra, arithmetic and geometry with
  applications ({W}est {L}afayette, {IN}, 2000), Springer, Berlin, 2004,
  pp.~703--723.

\bibitem[Voi21]{Voight21}
J.~Voight, \emph{Quaternion algebras}, Graduate Texts in Mathematics, vol. 288,
  Springer, Cham, 2021.

\end{thebibliography}
\end{document}